\documentclass{conm-p-l}
\usepackage{amsmath, amssymb}
\usepackage{xcolor}

\usepackage{graphicx} 
\usepackage{algorithmic}
\usepackage[]{algorithm2e}
\newtheorem{Theorem}{Theorem}[section]

\newtheorem{Lemma}[Theorem]{Lemma}
\newtheorem{Proposition}[Theorem]{Proposition}
\theoremstyle{definition}
\newtheorem{Definition}[Theorem]{Definition}
\theoremstyle{remark}
\newtheorem{rem}[Theorem]{Remark}
\newtheorem{ex}[Theorem]{Example}
\numberwithin{equation}{section}
\newcommand{\norm}[1]{\left\Vert\right\Vert}
\newcommand{\abs}[1]{\left\vert\right\vert}
\newcommand{\set}[1]{\left\{\right\}}
\newcommand{\R}{\mathbb R}
\newcommand{\N}{\mathbb N}

\newcommand{\C}{\mathcal{C}}
\newcommand{\F}{\mathcal{F}}

\newcommand{\A}{\mathcal{A}}

\newcommand{\rel}{\mathcal{R}}
\newcommand{\D}{\mathcal{D}}
\newcommand{\p}{\mathcal{P}}

\newcommand{\mcc}{\mathcal{C}}
\newcommand{\Hol}{\mathcal{H}}
\def\pathloc{ \operatorname{Paths}_{\operatorname{loc}} }

\definecolor{brightube}{rgb}{0.8, 0.4, 0.9}
\definecolor{ballblue}{rgb}{0.33, 0.55, 0.9}

\begin{document}
	
	\title[Diffeologies in infinite dimensional geometry and optimization]{On diffeologies  from infinite dimensional geometry to {PDE constrained optimization}}%
	

	\author{Nico Goldammer}
		\address[Nico Goldammer]{Helmut-Schmidt-University / University of the Federal Armed Forces Hamburg, Holstenhofweg~85, 22043 Hamburg, Germany}%
		\email{goldammer@hsu-hh.de}%
	\author{Jean-Pierre Magnot}
		\address[Jean-Pierre Magnot]{LAREMA - UMR CNRS 6093, Universit\'e d'Angers, 2 Boulevard Lavoisier 
		49045 Angers cedex 01 and Lyc\'ee Jeanne dArc 30 avenue de Grande Bretagne
		F-63000 Clermont-Ferrand
		http://orcid.org/0000-0002-3959-3443}%
	\email{jean-pierr.magnot@ac-clermont.fr}%
	\author{Kathrin Welker}
		\address[Kathrin Welker]{TU Bergakademie Freiberg, Akademiestr. 6, D-09599 Freiberg, Germany}%
	\email{Kathrin.Welker@math.tu-freiberg.de}%

	

	\begin{abstract}
		We review how diffeologies complete the settings classically used from  infinite dimensional geometry to {}{partial differential equations,} based on classical settings of functional analysis and with classical mapping spaces as key examples.   
	 {}{As the classical examples of function spaces, we deal with manifolds of mappings in Sobolev classes (and describe the ILB setting), jet spaces and spaces of triangulations, that are key frameworks for the two fields of applications of diffeologies that we choose to highlight: evolution equations and integrable systems, and optimization problems constrained by partial differential equations.}  
		
	\end{abstract}

	
	\maketitle

	\tableofcontents
	\section*{Introduction}
 Diffeological spaces and Fr\"olicher spaces are two frameworks, developed independently in the 20th century, in order to define smoothness on objects where basic differential properties cannot be stated in a standard way. {Motivated by technical uncapacities to formulate properly applied problems, many such settings developed independently, that are generalizations of the classical finite dimensional differential geometry, exist in the literature.} A non-exhaustive list of such settings can be found in e.g. \cite{St2011}. One global underlying problem is the following: 
 \vskip 6pt
	\centerline{\textit{{How to express variational problems in a setting where differentiation is not easy?}}}
	\vskip 6pt
	 One {has} immediately {in mind} spaces with singularities as elementary examples, such as the irrational torus or orbifolds where a rich differential geometry can be described by the use of diffeologies \cite{KIg}, but initialy the
	motivating (historical) examples for the development of diffeologies are groups of diffeomorphisms of non-compact manifolds and coadjoint orbits, in the works of French mathematician Jean-Marie Souriau (integrally written in French), reviewed and deeply expanded in \cite{Igdiff}. The same way, the motivations of the Swiss mathematician Alfred Fr\"olicher were certainly partly based on infinite dimensional considerations, till his book with Andreas Kriegl \cite{FK}. Therefore, motivated by historical considerations and by our own concerns, while most recent works in these two settings are developed in a background of algebraic or symplectic geometry (understood in a wide sense), we feel the need of an actualized presentation of diffeological spaces, and of Fr\"olicher spaces as a subcategory, based on infinite dimensional examples arising in functional analysis, infinite dimensional {geometry} 
	and numerical analysis. Indeed, since about {ten} years, under the impulsion of a new generation of motivated researchers, diffeologies have drastically increased their known technical capacities.
	 Examples and counter-examples have 
	 successfully {highlighted in} 
	 specific constructions that {point out} 
	 singular problems, 
	 drawing attention by non-specialists { to diffeologies} as a potential framework for some of their technical difficulties. This is to this class of potential readers that our text is primarily addressed. We are ourselves {primarily interested in} 
	 two very different backgrounds, both linked with functional analysis, and the use of diffeologies have imposed its necessity to us. 
	 {Indeed, after breakthrough in the theory of geometry and integrable systems \cite{Ma2013,MR2016},}
	  diffeological spaces are also {recently} considered in applied mathematics  (see, e.g., \cite{KW21}). {From the viewpoint of integrable systems, diffeologies enable to state smoothness, and hence well-posedness, on infinite dimensional (mostly algebraic) frameworks. From the viewpoint of shape optimization problems, diffeologies {}{propose} a more flexible framework than the one classically considered in the last decades, based on classical Riemannian infinite dimensional geometry.}

	One aim of this paper is to make a clear and understandable presentation of basic results on diffeologies and Frölicher spaces, based on the examples that we know better, but in a simple language, accessible to other researchers in the field of functional analysis and applied mathematics understood in a very wide sense. A further aim of this paper is to make a clear exposition of some basic concepts {where diffeologies are interrelated with topics} of functional analysis, {infinite dimensional geometry} and optimization for researchers whose basic knowledge remain e.g. in low dimensional topology or category theory, and who ignore such infinite dimensional settings. 
	
	The structure of the paper is as follows:
	\begin{itemize}
	\item A brief introduction in diffeological and Fr\"olicher spaces is given in section \ref{sec:DiffFroehl}. In particular, we present  and summarize the various tangent space definitions of diffeological spaces, which appear in the literature. {Other geometric objects generalized to diffeologies are also considered such as diffeological (Lie) groups, vector or principal (pseudo-) bundles, Riemannian metrics, groups of diffeomorphisms and automorphisms of various diffeological algebraic or set-theoric structures.}
		\item {Section \ref{diffeq} is concerned by (partial) differential equations in a two-fold way.  {We take profit of this exposition, in the two sections, to highlight or review on some tangent and cotangent structures that are well-known on (finite dimensional) manifolds and recovered in diffeological spaces with specific technical details.} In section \ref{s:reg}, we consider the problem of the existence of an exponential map in a diffeological Lie group, solving a first order differential equation, and as a by-product holonomy-type equations by stating an infinite diemensional Ambrose-Singer theorem.} 
		In section \ref{sec:Optimization} we present first order optimization techniques on diffeological spaces. Here, we concentrate on shape optimization, which has a wide range of application. 
		\item The paper ends with section \ref{sec:MappingSpaces} {where basic functional spaces, that is mapping spaces, are reviewed from the viewpoint of infinite dimensional geometry. Depending on the regularity required for mappings and on the source manifolds, one either observes Hilbert, Banach or ILB manifolds along the lines of \cite{Om}, or topological spaces which can be unserstood as diffeological spaces \cite{Ma2019}. Implicit functions, jet spaces, and spaces of triangulations are related examples where diffeologies apply in settings related to mapping spaces.}
	
\end{itemize}
	
	\section{A presentation of diffeological and Fr\"olicher spaces, with examples}
	\label{sec:DiffFroehl}
	
	\subsection{Diffeological and Fr\"olicher spaces}
	In this section we follow and complete 
	the expositions appearing in \cite{Ma2013,MR2019,MR2016}. The main reference for a comprehensive exposition on diffeologies is \cite{Igdiff}. This reference will be completed all along the text for specific concerns. 
	\begin{Definition} \label{d:diffeology} Let $X$ be a set.
		
		\noindent $\bullet$ A \textbf{p-parametrization} of dimension $p$ 
		on $X$ is a map from an open subset $O$ of $\R^{p}$ to $X$.
		
		\noindent $\bullet$ A \textbf{diffeology} on $X$ is a set $\p$
		of parametrizations on $X$ such that:
		
		\begin{itemize}
			\item For each $p\in\N$, any constant map $\R^{p}\rightarrow X$ is in $\p$;
			\item \label{d:local} For each arbitrary set of indexes $I$ and family $\{f_{i}:O_{i}\rightarrow X\}_{i\in I}$
			of compatible maps that extend to a map $f:\bigcup_{i\in I}O_{i}\rightarrow X$,
			if $\{f_{i}:O_{i}\rightarrow X\}_{i\in I}\subset\p$, then $f\in\p$.
			\item \label{d:compose} For each $f\in\p$, $f : O\subset\R^{p} \rightarrow X$, and $g : O' \subset \R^{q} \rightarrow O$, 
			in which $g$ is  
			a smooth map (in the usual sense) from an open set $O' \subset \R^{q}$ to $O$, we have $f\circ g\in\p$.
		\end{itemize} 
		
		\vskip 6pt If $\p$ is a diffeology on $X$, then $(X,\p)$ is
		called a \textbf{diffeological space} and, if $(X,\p)$ and $(X',\p')$ are two diffeological spaces, 
		a map $f:X\rightarrow X'$ is \textbf{smooth} if and only if $f\circ\p\subset\p'$. 
	\end{Definition} 
	
	The notion of a diffeological space is due to J.M. Souriau, see \cite{Sou}, who was inspired by the remarks and 
	 constructions from Chen in \cite{Chen}. A comprehensive exposition of basic concepts can be found in   \cite{Igdiff}. 
	\begin{Definition}
		Let $(X,\p)$ be a diffeological space. Let $\p'$ be another diffeology on $X.$ Then $\p'$ is a \textbf{subdiffeology} of $\p$ if and only if $\p' \subset \p$ or, in other words, if the identity mapping $(X,\p')\rightarrow (X,\p)$ is smooth.
	\end{Definition}
	\begin{ex}
		There exists on any nonempty set $X$ a diffeology, minimal for inclusion, made only of constant parametrizations. It is called the \textbf{discrete diffeology}. This is a subdiffeology of any diffeology on $X.$
	\end{ex}
\begin{ex}\label{ex:p(A)}
	{Let $X$ be a set. Given a non-empty family $A$ of functions $$f : O_f \rightarrow X$$ where $O_f$ is the domain of the function $f,$ which is an open subset of an Euclidian space, a direct application of Zorn's lemma proves that there exists a diffeology $\p(A),$ minimal for inclusion, that contains $A,$ in other words, for which $A$ is a family of smooth maps. This is the \textbf{diffeology generated by} $A.$}
\end{ex}
	The category of diffeological spaces is very large, and carries many different pathological examples even if it enables a very easy-to-use framework on infinite dimensional objects. therefore, a restricted category can be useful, the category of Fr\"olicher spaces.   
	This is for example the framework chosen in \cite{Can2020}.
	
	\begin{Definition} 
		A \textbf{Fr\"olicher} space is a triple $(X,\F,\C)$ such that
		
		- $\C$ is a set of paths $\R\rightarrow X$,
		
		- $\F$ is the set of functions from $X$ to $\R$, such that a function
		$f:X\rightarrow\R$ is in $\F$ if and only if for any
		$c\in\C$, $f\circ c\in C^{\infty}(\R,\R)$;
		
		- A path $c:\R\rightarrow X$ is in $\C$ (i.e. is a \textbf{contour})
		if and only if for any $f\in\F$, $f\circ c\in C^{\infty}(\R,\R)$.
		
		\vskip 5pt If $(X,\F,\C)$ and $(X',\F',\C ')$ are two
		Fr\"olicher spaces, a map $f:X\rightarrow X'$ is \textbf{smooth}
		if and only if $\F'\circ f\circ\C\subset C^{\infty}(\R,\R)$.
	\end{Definition}
	
 This definition first appeared in {\cite{Fr1982}, see e.g.} \cite{FK}, but the actual terminology 
	was fixed to our knowledge in Andreas Kriegl and Peter Michor's book \cite{KM}, and independently by Paul Cherenack in \cite{Ch1998,Ch1999}.
	
	\vskip 6pt
	
	The comparison of these two frameworks does need to be exposed for a complete review.  Its first steps were published in \cite{Ma2006-3}; the reader can 
	also see \cite{Ma2013,Ma2018-2,MR2016,Wa} for extended expositions.
	In particular, it is explained in \cite{MR2016} that 
	{\em Diffeological, Fr\"olicher and Gateaux smoothness are the same notion if we 
		restrict ourselves to a Fr\'echet context,} in a way that we explain here. 
	
	For this, we first need to analyze how we generate a Fr\"olicher or a diffeological space, that is, how we implement a Fr\"olicher or a diffeological structure on a given set $X.$
	
	Any family of maps $\F_{g}$ from $X$ to $\R$ generates a 
	Fr\"olicher structure $(X,\F,\C)$ by setting, after \cite{KM}:
	
	- $\C=\{c:\R\rightarrow X\hbox{ such that }\F_{g}\circ c\subset C^{\infty}(\R,\R)\}$
	
	- $\F=\{f:X\rightarrow\R\hbox{ such that }f\circ\C\subset C^{\infty}(\R,\R)\}.$
	
	We call $\F_g$ a \textbf{generating set of functions}
	for the Fr\"olicher structure $(X,\F,\C)$. One easily see that
	$\F_{g}\subset\F$. 
	A Fr\"olicher space $(X,\F,\C)$
	carries a natural topology, the pull-back topology of
	$\R$ via $\F$. 
	
	The same way, one can start alternatively from a generating set of contours $\C_g,$ and set:
	
	- $\F=\{f:X\rightarrow\R\hbox{ such that }f\circ\C_g\subset C^{\infty}(\R,\R)\}$
	 
	- $\C=\{c:\R\rightarrow X\hbox{ such that }\F\circ c\subset C^{\infty}(\R,\R)\}.$
	
	In the case of a finite dimensional
	differentiable manifold $X$ we can take $\F$ as the set of all smooth
	maps from $X$ to $\R$, and $\C$ the set of all smooth paths from
	$\R$ to $X.$ Then, the underlying topology of the
	Fr\"olicher structure is the same as the manifold topology
	\cite{KM}. 
	
	We also remark that if $(X,\F, \C)$ is a Fr\"olicher space, we can
	define a natural diffeology on $X$ by using the following family
	of maps $f$ defined on open domains $D(f)$ of Euclidean spaces, see \cite{Ma2006-3}:
	$$
	\p_\infty(\F)=
	\coprod_{p\in\N^*}\{\, f: D(f) \rightarrow X\, | \, D(f) \hbox{ is open in } \R^p \hbox{ and } \F \circ f \in C^\infty(D(f),\R) \}.$$
	
	If $X$ is a finite-dimensional differentiable manifold, setting $\F=C^\infty(X,\R),$ this diffeology is
	called the { \em nebulae diffeology}, see e.g. \cite{Igdiff}. Now,
	we can easily show the following:
	
	\begin{Proposition} \label{fd} \cite{Ma2006-3}
		Let$(X,\F,\C)$
		and $(X',\F',\C')$ be two Fr\"olicher spaces. A map $f:X\rightarrow X'$
		is smooth in the sense of Fr\"olicher if and only if it is smooth for
		the underlying diffeologies $\p_\infty(\F)$ and $\p_\infty(\F').$
	\end{Proposition}
	
	Thus, Proposition \ref{fd} and the foregoing remarks imply that 
	the following implications hold:
	\vskip 12pt
	
	\begin{tabular}{ccccc}
		smooth manifold  & $\Rightarrow$  & Fr\"olicher space  & $\Rightarrow$  & diffeological space
	\end{tabular}
	{
		\vskip 12pt \noindent These implications can be analyzed in a refined way. The reader is
		referred to the Ph.D. thesis \cite{Wa} for a deeper analysis of 
		them. 
		
		\begin{rem}
			The set of contours $\C$ of the Fr\"olicher space $(X,\F,\C)$ does 
			not give us a diffeology, because a diffeology needs to be stable 
			under restriction of domains. In the case of paths in $\C$ the 
			domain is always $\R$ whereas the domain of 1-plots can (and has 
			to) be any interval of $\R.$ However, $\C$ defines a ``minimal 
			diffeology'' $\p_1(\F)$ whose plots are smooth parametrizations 
			which are locally of the type $c \circ g,$ in which 
			$g \in \p_\infty(\R)$  and $c \in \C.$ Within this setting, we can 
			replace $\p_\infty$ by $\p_1$ in Proposition $\ref{fd}$. The main
			technical tool needed to discuss this issue is Boman's theorem 
			\cite[p.26]{KM}.
			Related discussions are in \cite{Ma2006-3,Wa}.
		\end{rem}
		After this remark, one can push further the ``restriction'' procedure, for any diffeology $\p$ on $X.$
		
		\begin{Theorem} \label{th:dim}
			Let $\p$ be a diffeology on $X.$
			Let $\F(\p)$ be the set of smooth functions from $(X,\p)$ to $(\R,\p_\infty(\R))$ where $\R$ is understood as a smooth manifold. Then: 
			\begin{itemize}
				\item there exists a nebulae diffeology $\p_\infty$ generated by $\F(\p),$
				\item $\p\subset \p_\infty$
				\item $\forall k \in \N^*,$ there exists a \textbf{$n-$ dimensional diffeology} $\p_k \subset \p$ made of plots $p \in \p$ which factors tthrough $$ D(\p) \rightarrow \R^k \rightarrow X$$
				\item $\forall k \in \N,$ the set of smooth maps from $(X,\p_k)$ to $(\R,\p_\infty(R))$ coincides with $\F(\p).$
			\end{itemize}
		\end{Theorem}
		This construction for $\p=\p_\infty$ can be found in \cite{Ma2013}.
		It extends trivially to any diffeology $\p.$ By this theorem, one can easily see that the categories of Fr\"olicher spaces and diffeological spaces do not coincide. 
		We now have to precise:
		\begin{Definition}
			\cite{Wa}
			If $\p=\p_\infty,$ the diffeology $\p$ is called reflexive (and it is the nebulae diffeology of a Fr\"olicher space).
		\end{Definition}  
	
	\begin{rem}
		Given a diffeological space $(X,\p)$ there is a natural topology on $X$ called $D-$topology in \cite{CSW} which is the final topology for the set of maps $p \in \p.$ This topology can be rather complicated and actually not-so-studied. To our best knowledge, this problem of topology associated to a diffeology is better studied actually in the context of Fr\"olicher spaces, see e.g. \cite{FK,KM} for a serie of technical results.
	\end{rem}
	Let us now give a {class of} examples of diffeologies.
	 
	 \begin{ex}
	 	Let $k \in \N.$ Let $(X,\F,\C)$ be a Fr\"olicher space. The $C^k-$diffeology of $(X,\F,\C)$ is the diffeology $\p_{(k)}$ defined by: $$\forall d \in \N^*, \, \forall O \hbox{ open subset of } \R^d \quad {\p_{(k)}}_O = \left\{ p : O \rightarrow X \, | \, \F \circ p \subset C^k(O,\R) \right\}$$
	 		and 
	 		$$\p_{(k)} = \bigcup_{d \in \N^*} \bigcup_{O \hbox{ open in }\R^d} {\p_{(k)}}_O.$$
	 		This example is given in \cite{Igdiff} for a finite dimensional manifold. The extension to Fr\"olicher structures is straightforward. 
	 		
	 \end{ex}

		\subsection{Diffeologies on spaces of mappings, products, quotients, subsets and algebraic structures}
		\begin{Proposition} \label{prod1} \cite{Sou,Igdiff} Let $(X,\p)$ and $(X',\p')$
			be two diffeological spaces. There exists a diffeology 
			$\p\times\p'$ on
			$X\times X'$  made of plots $g:O\rightarrow X\times X'$
			that decompose as $g=f\times f'$, where $f:O\rightarrow X\in\p$
			and $f':O\rightarrow X'\in\p'$. We call it the \textbf{product diffeology}, 
			and  this construction extends to an infinite (maybe not countable) product.
		\end{Proposition}

		We apply this result to the case of Fr\"olicher spaces and we 
		derive (compare with \cite{KM}) the following:
		
		\begin{Proposition} \label{prod2} 
			Let $(X,\F,\C)$ and $(X',\F',\C')$ be two Fr\"olicher spaces 
			equipped with their natural diffeologies $\p$ and $\p'$ . There is 
			a natural structure of Fr\"olicher space on $X\times X'$ which 
			contours $\C\times\C'$ are the 1-plots of $\p\times\p'$. 
		\end{Proposition}
		
		We can also state the above result for infinite products;
		we simply take Cartesian products of the plots, or of the contours.
		
		\smallskip
		
		Now we consider quotients after \cite{Sou} and 
		\cite[p. 27]{Igdiff}: Let $(X,\p)$ be 
		a diffeological space, and let $X'$ be a set. Let 
		$f:X\rightarrow X'$ be a map. We define the \textbf{push-forward 
			diffeology}, noted by $f_*(\p),$ as the coarsest (i.e. minimal for inclusion) among 
		the diffologies on $X'$, which contains $f \circ \p,$ that is, for which $f$ is smooth. {This construction is a specification of example \ref{ex:p(A)}, where the generating family $A$ is set as $A=f \circ \p.$ A smooth map $f:(X,\p) \rightarrow (X',\p ')$ between two diffeoogical spaces is called a \textbf{subduction} if it is surjective anf if $f_*(\p) = \p'.$} Conversly, if $(X',\p')$ is a diffeological space and if $X$ is a set, if $f:X\rightarrow X',$ the \textbf{pull-back diffeology} noted by $f^*(\p')$ is the diffeology on $X,$ maximal for inclusion, for which $f$ is smooth. 
		
		\begin{Proposition} \label{quotient} \cite{Sou,Igdiff}
			Let $(X,\p)$ b a diffeological space and $\rel$ an equivalence 
			relation on $X$. Then, there is a natural diffeology on $X/\rel$, 
			noted by $\p/\rel$, defined as the push-forward diffeology on 
			$X/\rel$ by the quotient projection $X\rightarrow X/\rel$. 
		\end{Proposition}
		
		Given a subset $X_{0}\subset X$, where $X$ is a Fr\"olicher space
		or a diffeological space, we equip $X_{0}$ with structures 
		induced by $X$ as follows:
		\begin{enumerate}
			\item If $X$ is equipped with a diffeology $\p$, we define
			a diffeology $\p_{0}$ on $X_{0}$ called the \textbf{subset or trace 
				diffeology}, see \cite{Sou,Igdiff}, by setting 
			\[ 
			\p_{0}=
			\lbrace p\in\p \hbox{ such that the image of }p\hbox{ is a subset 
				of } X_{0}\rbrace\; .
			\]
			\item If $(X,\F,\C)$ is a Fr\"olicher space, we take as a 
			generating set of maps $\F_{g}$ on $X_{0}$ the restrictions of the 
			maps $f\in\F$. In this case, the contours (resp. the induced 
			diffeology) on $X_{0}$ are the contours (resp. the plots) on $X$ 
			whose images are a subset of $X_{0}$.
		\end{enumerate}
		\begin{ex}\label{ex:YpowerX}
			Let us consider the infinite product diffeology on $Y^X.$ This is the largest diffeology for which the evaluation maps $\operatorname{ev}_x$ are smooth for each $x \in X.$ Therefore, any other diffeology on any subset of $Y^X$ {for which evaluation maps are smooth} is a subdiffeology of the subset diffeology inherited from the this product diffeology.
		\end{ex}
			\begin{ex}
			Let $M$ be a smooth (Riemannian) manifold. Then, the set $M^\N$ of sequences on $M$ is a topological space with open sets of the form $$O_0 \times O_1 \times ... \times O_k \times ...$$ where the family $\{O_k, k \in \N\}$ is a family of open subsets of $M$ such that $O_k=M$ except for a finite set of indexes $k.$ With this topology, it seems difficult to define an atlas on $M^\N$ while an intuitive ``natural'' differentiation holds by considering differentiation on each component of the infinite product. By considering a diffeology $\p$ on $M$ such that $\F(\p)=C^\infty(M,\R)$ in the usual sense, the diffeology $\p$ also defines a diffeology on $M^\N$ {by (countable) infinite dimensional product diffeology along the lines of Example \ref{ex:YpowerX}}  that encodes this ``natural differentiation''. {An example of such a diffeology is obtained by setting $\p=\p_\infty(M).$} Following Proposition \ref{prod2}, it is an easy exercise to check that the n-dimensional diffeology $\p_n$ related to $\p$ generates also a n-dimensional diffeology on the infinite product. 
			
			Moreover, one can also consider subsets of $M^\N$ which inherit a subset diffeology, such as the so called \emph{marked (infinite) configuration spaces} (see e.g. \cite{AKLU2000}) that we note here by 
			$$O\Gamma_M =\left\{ (x_n)\in M^\N \, | \, \forall (m,n)\in \N^2, m \neq n \Leftrightarrow x_n \neq x_m \right.$$
			$$\left. \hbox{ and for all bounded set } B \subset M, |\{x_n \, | \, n \in \N\}\cap B| < +\infty \right\}.$$
			It is an easy exercise to prove that he differentiable structure on $O\Gamma_M,$ defined by another generatized notion of differentiability called Sikorski differentiability (see e.g. \cite{St2011} for an overview), fits here with our diffeology, once we have remarked that each point $x_n \in M$ can be identified with its Dirac measure $\delta_{x_n}$ acting on $C^\infty(M,\R).$ 
			
			Associated with marked infinite configuration spaces, the (unmarked) \emph{infinite configuration space} $\Gamma_M$ was studied by the same research group, see e.g. \cite{ADL2001}, can be understood as the space of marked configurations up to reindexation. More precisely, the infinite symmetric group $\mathfrak{G}_\infty$ made of bijections of $\N$ is acting on the indexes of the sequences of $O\Gamma_M,$ and $$\Gamma_M = O\Gamma_M /   \mathfrak{G}_\infty.$$ Equipped with its quotient diffeology, this space inherits the same structure for generalized differentiability as the one described in the original works. 
						Other similar examples can be found in \cite{Ism}.
	
		\end{ex}

		\begin{ex}\cite{Ma2020-3}
			Let $Y$ be a smooth complete manifold (not necessarily finite dimensional) and let $(X,\p)$ be a diffeological space such that $X \subset Y$ as a set, and $\p \subset \p_\infty(Y)$ (in other words, the canonical inclusion map is smooth for the nebulae diffeology on $Y$). Then we define $\C(X,Y)$ as the set of sequences in $X$ that converge in $Y,$ in other words, the subset of $X^\N$ made of sequences that are Cauchy for the topology of $Y.$ Then, 
			\begin{itemize}
				\item the limit $$
				\lim: \C(X,Y) \rightarrow Y$$ is not smooth when $\C(X,Y)$ is equipped with the subset topology inherited from the infinite product diffeology of $X^\N$, and when $Y$ is equipped with $\p_\infty(Y),$ as soon as there exists a non constant smooth path in $X.$
				\item  $\lim^*(\p_\infty(Y))$ may make the canonical projection (evaluation) maps $$ \C(X,Y) \rightarrow X$$ nonsmooth, consider for example $Y = \R$ and $X = \mathbb{Q}$ equipped with its subset diffeology.
			\end{itemize} 
		Therefore, there is a true necessity to define the \emph{Cauchy diffeology} on $\C(X,Y)$ for which the limit map and the evaluation maps are smooth in $\C(X,Y).$  
		\end{ex}
		Our last general construction is the so-called functional 
		diffeology. 
		{
			Its existence implies the following crucial fact: the category of 
			diffeological spaces is Cartesian 
			closed.} Our discussion follows \cite{Igdiff}. 
		Let $(X,\p)$ and $(X',\p')$ be diffeological spaces. 
		Let $M \subset C^\infty(X,X')$ be a set of smooth maps. 
		The \textbf{functional diffeology} on $S$ is the diffeology $\p_S$
		made of plots 
		$$ \rho : D(\rho) \subset \R^k \rightarrow S$$
		such that, for each $p \in \p$, the maps 
		$\Phi_{\rho, p}: (x,y) \in D(p)\times D(\rho) \mapsto 
		\rho(y)(x) \in X'$ are plots of $\p'.$ 
		We have, see \cite[Paragraph 1.60]{Igdiff}:
		
		\begin{Proposition} \label{cvar} 
			Let $X,Y,Z$ be diffeological spaces. Then, 
			$$
			C^\infty(X\times Y,Z) = C^\infty(X,C^\infty(Y,Z)) = 
			C^\infty(Y,C^\infty(X,Z))
			$$
			as diffeological spaces equipped with functional diffeologies.
		\end{Proposition}
		
		The next property will be useful for us:
		\begin{Proposition}
			The composition of maps is smooth for functional diffeologies. 
		\end{Proposition}
		
			Now, given an algebraic structure, we can define a
		corresponding compatible diffeological (resp. Fr\"olicher) structure, see 
		for instance \cite{Les}. For example, see \cite[pp. 66-68]{Igdiff},
		if $\R$ is equipped with its canonical diffeology (resp. Fr\"olicher
		structure), we say that an $\R-$vector space equipped with a 
		diffeology (resp. Fr\"olicher structure) is a diffeological (resp. Fr\"olicher) vector space if addition and 
		scalar multiplication are smooth. We state:

		\begin{Definition}
			Let $G$ be a group equipped with a diffeology (Fr\"olicher 
			structure). We call it a \textbf{diffeological (Fr\"olicher) group} 
			if both multiplication and inversion are smooth.
		\end{Definition}
	
	The same way, one can define diffeological fields, rings, algebras, groupoids, actions and so on.
	
	\subsection{The principle of pull-back on parametrizations}
	We explain here, with non-categorical vocabulary, some of the constructions that are necessary for the difinition of a \textbf{colimit,} which is often used to define geometric objects on a diffeology. We perform this contruction under three examples: the internal or kinematic tangent space, the Riemannian metric and the Haussdorf volume.  

One has to remark that the plots of a diffeology are behaved like charts on a manifold through the property \ref{d:local} in Definition \ref{d:diffeology}. The notion of parametrization splits with the notion of chart at the point where a chart defines an open neighborhood of each point of its image set, which is not the case for a plot of a given diffeology. Moreover, for a given diffeological space $(X,\p),$ the diffeology carries a notion of a given plot dimension in two ways:
\begin{itemize}
	\item  by the dimension of the domain of the plot $p,$
	\item and by the filtration of diffeologies defined in Theorem \ref{th:dim}:
	$$\p_1 \subset \p_2 \subset \cdots \subset \p.$$ 
\end{itemize}
Indeed, rephrasing the definition of $\p_n, $ a plot $p \in \p$ lies in $\p_n$ if there exists $g \in \p_\infty(\R^n)$ and $h$ a smooth map from $(\R^n,\p_\infty(\R^n))$ to $(X,\p)$ such that $p = h \circ g.$ 

Now, let us consider a smooth manifold $M$ equipped with its nebulae diffeology. Some geometric constructions (differential forms, Riemannian metrics, etc.) can be pulled back from $M$ to the domain of any plot $p.$ Therefore, there is a \emph{pull-back principle} that we can apply to define the same contructions on the domain of any plot of a diffeology $\p$ provided that they are stable under pull-back via composition of smooth maps. Let us develop in few first examples how basic  constructions on finite dimnsional manifolds extend to diffeological spaces through this principle.
\subsubsection{The kinematic (or internal) tangent cone}\label{s:internal tangent}
We describe here the construction first established in \cite{Les} for diffeological groups, in \cite{DN2007-1} for Fr\"olicher spaces, and extended in \cite{CW} for diffeologies, in a categorical vocabulary. Our {}{aim} is to express the same construction without the assumption that the reader is familiar with the categorical framework. 

Let $(X,\p)$ be a diffeological space. The domain $D(p)$ of each plot $p$ of $\p$ can be considered as a smooth manifold, and first objects of interest are tangent vectors in the tangent space $TD(p).$ They are understood as germs of smooth path $\frac{d\gamma}{dt}|_{t=0}$ where $\gamma \in C^\infty(\R,D(p)).$  Let $x \in X$ and let us consider 
$$\mathcal{C}_x = \left\{ \gamma \in C^\infty(\R,X) \, | \, \gamma(0)=x\right\}.$$ For each $p \in \p,$ we also define
$$\mathcal{C}_{x,p} = \left\{\gamma \in C^\infty(\R,D(p)) \, | \, p \circ \gamma(0) = x\right\}.$$ This set of smooth paths passing at $x$ enables to define the kinematic set 

$$\mathcal{K}_x = \coprod_{p \in \p} \left\{ \frac{d \gamma}{dt}|_{t=0} \, | \, \gamma \in \C_{x,p} \right\} = \coprod_{p \in \p} \coprod_{x_0 \in p^{-1}(x)} T_{x_0}D(p).$$

Therefore, we identify $(X_1,X_2) \in \mathcal{K}_x^2, $ where $X_1 = \frac{d \gamma_1}{dt}|_{t=0} \in T_{x_1}D(p_1)$ and  $X_2 = \frac{d \gamma_2}{dt}|_{t=0} \in T_{x_2}D(p_2)$ if there exists a $p_{3} \in \p$ and $(\gamma_{3,1},\gamma_{3,2}) \in C_{x,p_3}$ such that
\begin{equation} \label{id-germs}  \left\{\begin{array}{l} \forall i \in \{1,2\}, p_i \circ \gamma_i = p_{3,i} \circ \gamma_{3,i} \\
	\frac{d\gamma_{3,1}}{dt}|_{t=0} = \frac{d\gamma_{3,2}}{dt}|_{t=0}  
	\end{array}\right.\end{equation}
This identification is reflexive, symmetric, but not transitive, as is shown in the following counter-example:
\begin{ex} \label{spagh}
	Let $$ X= \{(x,y,z) \in \R^3 \, | \, yz = 0\}.$$
	We equip $X$ with its subset diffeology, inherited from the nebulae diffeology of $\R^3.$ Let us conside the paths $$\gamma_1(t) = (t,t^2,0)$$ and $$\gamma_2(t)=(t,0,t^2).$$
	The natural intuition (which will be shown to have a defect later in the exposition, for another diffeology on $X$) says that $$\frac{d\gamma_{1}}{dt}|_{t=0} = \frac{d\gamma_{2}}{dt}|_{t=0} = (1,0,0),$$
	 \emph{but there is no parametrization $p_3$ at $x=(0,0,0)$ that fulfills (\ref{id-germs}).}
	Indeed, the diffeology of $X$ is generated by the push-forwards of the plots of $\R^2$ to $X$ by the maps $(x,y) \mapsto (x,y,0) $ and $(x,y \mapsto (x,0,y).$ In order to identify the two germs, one has to consider an intermediate path $\gamma_{1.5}(t) = (t,0,0).$ 
		   
\end{ex}
 Therefore, we define 
 
 \begin{Definition}
 	We define the equivalence relation $\sim$ on $\mathcal{K}_x$ as follows: 
 	$\forall (X_1,X_2) \in \mathcal{K}_x, u_1 \sim u_2$ if and only if one of the two following conditions is fulfilled: 
 	\begin{enumerate}
 		\item $$\exists (\gamma_1,\gamma_2) \in \left(\coprod_{p \in \p}\C_{x,p}\right)^2,  \frac{d\gamma_{1}}{dt}|_{t=0} = u_1 \in TD(p_1), \frac{d\gamma_{2}}{dt}|_{t=0} = u_2 \in TD(p_2)$$ and also $p_{3} \in \p$ and $(\gamma_{3,1},\gamma_{3,2}) \in C_{x,p_3}$ such that condition (\ref{id-germs}) applies. 
 		\item there exists a finite sequence $(v_1,...,v_k)\in \mathcal{K}_x^k$ such that $v_1 = u_1,$ $v_k = u_2,$ and such that condition (1) applies to $v_i$ and $v_{i+1}$ for each index $i \in \N_{k-1}.$
 	\end{enumerate} 

 \end{Definition}
{\begin{Definition}
	The \emph{internal} or \emph{kinetic tangent cone} of $X$ at $x \in X$ is defined by $${}^iT_xX = \mathcal{K}_x/\sim,$$
	The space ${}^iT_xX$ is endowed by the push-forward of the functional diffeology on $\mathcal{C}_x.$
	\end{Definition}}
 The elements of ${}^iT_xX$ are called \emph{germs of paths} on $X$ at $x.$ In fact, we will produce later in our exposition other tangent spaces, so that we avoid the terminology ``tangent vector'' which {may be misleading}.   
	\subsubsection{Differential forms and de Rham complex}
	\begin{Definition} \cite{Sou}
		Let $(X,\p)$ be a diffeological space and let $V$ be a vector space equipped with a differentiable structure. 
		A $V-$valued $n-$differential form $\alpha$ on $X$ (noted $\alpha \in \Omega^n(X,V))$ is a map 
		$$ \alpha : \{p:O_p\rightarrow X\} \in \p \mapsto \alpha_p \in \Omega^n(p;V)$$
		such that 
		
		$\bullet$ Let $x\in X.$ $\forall p,p'\in \p$ such that $x\in Im(p)\cap Im(p')$, 
		the forms $\alpha_p$ and $\alpha_{p'}$ are of the same order $n.$ 
		
		$\bullet$ Moreover, let $y\in O_p$ and $y'\in O_{p'}.$ If $(X_1,...,X_n)$ are $n$ germs of paths in 
		$Im(p)\cap Im(p'),$ if there exists two systems of $n-$vectors $(Y_1,...,Y_n)\in (T_yO_p)^n$ and $(Y'_1,...,Y'_n)\in (T_{y'}O_{p'})^n,$ if $p_*(Y_1,...,Y_n)=p'_*(Y'_1,...,Y'_n)=(X_1,...,X_n),$
		$$ \alpha_p(Y_1,...,Y_n) = \alpha_{p'}(Y'_1,...,Y'_n).$$
		
		We note by $$\Omega(X;V)=\oplus_{n\in \mathbb{N}} \Omega^n(X,V)$$ the set of $V-$valued differential forms.  
	\end{Definition}
	
	With such a definition, we feel the need to make two remarks for the reader:
	
	$\bullet$ If there does not exist $n$ linearly independent vectors $(Y_1,...,Y_n)$
	defined as in the last point of the definition, $\alpha_p = 0$ at $y.$
	
	$\bullet$ Let $(\alpha, p, p') \in \Omega(X,V)\times \p^2.$ 
	If there exists $g \in C^\infty(D(p); D(p'))$ (in the usual sense) 
	such that $p' \circ g = p,$ then $\alpha_p = g^*\alpha_{p'}.$ 
	
	\vskip 12pt
	\begin{Proposition}
		The set $\p(\Omega^n(X,V))$ made of maps $q:x \mapsto \alpha(x)$ from an open subset $O_q$ of a 
		finite dimensional vector space to $\Omega^n(X,V)$ such that for each $p \in \p,$ $$\{ x \mapsto \alpha_p(x) \} \in C^\infty(O_q, \Omega^n(O_p,V)),$$
		is a diffeology on $\Omega^n(X,V).$  
	\end{Proposition}
	
	Working on plots of the diffeology, one can define the product and the differential of differential forms, which have the same properties as  
	the product and the differential of differential forms. 
	
	\begin{Definition}
		Let $(X,\p)$ be a diffeological space. 
		
		\noindent
		$\bullet$ $(X,\p)$ is \textbf{finite-dimensional} if and only if $$\exists n_0\in\mathbb{N},\quad \forall n\in \mathbb{N}, \quad n\geq n_0 \Rightarrow \operatorname{dim}(\Omega^n(X,\mathbb{R}))=0.$$
		Then, we set $$\operatorname{dim}(X,\p)=\max\{n\in \mathbb{N}| \operatorname{dim}(\Omega^n(X,\mathbb{R}))>0\}.$$
		\noindent
		$\bullet$ If not, $(X,\p)$ is called \textbf{infinite dimensional}.
	\end{Definition}
	
	Let us make a few remarks on this definition.
	If $X$ is a manifold with $\operatorname{dim}(X)=n$ {(in the classical sense)  the nebulae diffeology $\p_\infty(X)$} is such that {$$\operatorname{dim}(X,\p_\infty(X))=n.$$}
	Now, if $(X,\F,\C)$ is the natural Fr\"olicher structure on $X,$ take $\p_1$ generated by the maps of the type
	$g\circ c$, where $c\in \C$ and $g$ is a smooth map from an open subset of a finite dimensional space to $\mathbb{R}.$
	This is an easy exercise to show that $$\operatorname{dim}(X,\p_1)=1.$$
	This first point shows that the dimension depends on the diffeology considered. Now, we remark that $\F$
	is the set of smooth maps $(X,\p_1)\rightarrow \mathbb{R},$

	This leads to the following definition, since $\p(\F)$ is clearly the diffeology with the biggest dimension associated to $(X,\F,\C)$:
	
	\begin{Definition}
		The \textbf{dimension} of a Fr\"olicher space $(X,\F,\C)$ is the dimension of the diffeologial space $(X,\p(\F)).$
	\end{Definition}

	This definition totally fits with the following: 
	\begin{Definition}
		Let $(X,\p)$ be a diffeological space {which is not totally disconnected, that is, for which $\p$ has non-constant plots.} We define the dimension of $(X,\p)$ as 
		$$\operatorname{dim}(X,\p) = min \left\{ n \in \N^* \, | \, {}{\p_n} = \p \right\}.$$
	\end{Definition}

	\subsubsection{(Pseudo-)Riemannian metrics {on diffeological spaces}} \label{s:rmdiff}
	We now go {to} the extension of the basic structures of Riemannian manifolds to diffeological spaces.

	\begin{Definition} \label{d:RM}
		Let $(X,\p)$ be a diffeological space.
		An \textbf{internal Riemannian metric} $g$ on $X$ (noted $g \in Met(X)$) is a map 
		$$  g: \{p:O_p\rightarrow X\} \in \p \mapsto g_p$$
		such that 
		
		$\bullet$ $x \in O_p \mapsto g_p(x)$ is a (smooth) metric on $TO_p$  
		
		$\bullet$ Moreover, let $y\in O_p$ and $y'\in O_{p'}$ such that $p(y)=p'(y').$ If $(X_1,X_2)$ is a couple of germs of paths in 
		$Im(p)\cap Im(p'),$ if there exists two systems of $2-$vectors $(Y_1,Y_2)\in (T_yO_p)^2$ and $(Y'_1,Y'_2)\in (T_{y'}O_{p'})^2,$ if $p_*(Y_1,Y_2)=p'_*(Y'_1,Y'_2)=(X_1,X_2),$
		$$ g_p(Y_1,Y_2) = g_{p'}(Y'_1,Y'_2).$$
		
		\noindent
		$(X,\p,g)$ is a \textbf{internal Riemannian diffeological space} if $g$ is a metric on $(X,\p).$ 
	\end{Definition}
	For any germ of path $X$ we note $||X|| = \sqrt{g(X,X)}$ This notation is independent on the chosen plot of the diffeology.
	\begin{Definition}
		We call \textbf{arc length} the map  $L : C^\infty([0;1], X) \rightarrow \mathbb{R}_+$ defined by $$L(\gamma) = \int_0^1 ||\dot\gamma(t)|| dt.$$
		Let $(x, y) \in X^2.$ We define $$ d(x,y) = \inf \left\{ L(\gamma) | \gamma(0) = x \wedge \gamma(1)=y\right\}$$
		and we call \textbf{Riemannian pseudo-distance} the map $d : X \times X \rightarrow \R_+$ that we have just described. 
	\end{Definition}
	
	The following proposition justifies the terms ``pseudo-distance'': 
	\begin{Proposition} \cite{Ma2019}
		\begin{enumerate}
			\item \label{d1} $\forall x \in X, d(x,x) = 0.$
			\item \label{d2} $\forall (x,y) \in X^2, d(x,y) = d(y,x)$
			\item \label{d3} $\forall (x,y,z \in X^3, d(x,z) \leq d(x,y) + d(y,z).$  
		\end{enumerate}
	\end{Proposition}
	
	
	One could wonder whether $d$ is a distance or not, i.e. if we have the stronger property: 
	$$ \forall (x,y) \in X^2, d(x,y)=0 \Leftrightarrow x=y.$$
	Unfortunately, it seems to appear in examples arising from infinite dimensional geometry that there can 
	have a distance which equals to 0 for $x \neq y.$ This is what is described on a weak Riemannian metric 
	of a space of proper immersions in the work of Michor and Mumford \cite{MichMum}. 
	Moreover, the D-topology is not the topology defined by the pseudo-metric $d.$ All these facts, which show 
	that the situation on Riemannian Fr\"olicher spaces is very different from the one known on manifolds, are checked in the following (toy) example.
	
	\begin{rem}
		Let $Y = \coprod_{i \in \N^*} \R_i$ where $\R_i$ is the $i-$th copy of $\R,$ equipped with its natural scalar product.
		Let $\mathcal{R}$ be the equivalence relation $$x_i \mathcal{R} x_j \Leftrightarrow  \left\{ \begin{array}{l}(x_i \notin ]0;\frac{1}{i}[ \wedge x_j \notin ]0;\frac{1}{j}[) \Rightarrow \left\{ \begin{array}{lr} x_i = x_j & \hbox{if } x_i\leq 0 \\
				x_i + 1- \frac{1}{i}  = y_j + 1 - \frac{1}{j} & \hbox{ if } x_i \geq \frac{1}{i} \end{array} \right. \\ (x_i \in ]0;\frac{1}{i}[ \vee x_j \in ]0;\frac{1}{j}[) \Rightarrow i = j \wedge x_i = x_j  \end{array} \right. $$
		Let $X = Y / \mathcal{R}.$ This is obviously a 1-dimensional Riemannian diffeological space. Let $\bar{0}$ be the class of $0 \in\R_1,$ and let $\bar{1}$ 
		be the class of $1 \in \R_1.$  Then $d(\bar{0}, \bar{1}) = 0.$ This shows that $d$ is not a distance on $X.$  In the $D-$topology, $\bar{0}$ and $\bar{1}$ respectively have the following disconnected neighborhoods: 
		$$U_{\bar{0}} = \left\{ \bar{x_i} | x_i < \frac{1}{2i} \right\}$$
		and $$U_{\bar{1}} =\left\{\bar{x_i} | x_i > \frac{1}{2i} \right\}.$$
		This shows that $d$ does not define the $D-$topology.  
	\end{rem}

	\subsection{Group of diffeomorphisms and linear group: two examples of diffeological groups}
	Let us start with the generalization of the notion of diffeomorphism \cite{Sou}.
	 \begin{Definition}
	 	Let $(X,\p)$ be a diffeological space. A (set theoric) bijection on $X$ is a \emph{diffeomorphism} if both $f$ and $f^{-1}$ are smooth. We note by $\operatorname{Diff}(X)$ the group of diffeomorphisms of $X.$
	 \end{Definition}
 \begin{Proposition}
 	Let us equip $C^\infty(X,X)$ with its functional diffeology $\p.$ We define the maps
 	$$ i : f \in \operatorname{Diff}(X) \mapsto f\in C^\infty(X,X)$$
 	and 
 	$$ (.)^{-1} :  f \in \operatorname{Diff}(X) \mapsto f^{-1}\in C^\infty(X,X).$$ 
 	Then $(\operatorname{Diff}(X), i^*(\p)\cap ((.)^{-1})^*(\p))$ is a diffeological group.
 \end{Proposition}
 
 \begin{rem}
 	When $X$ is a compact boundaryless manifold, the group $\operatorname{Diff}(X)$ is an open submanifold of $C^\infty(X,X)$ 
 	in the sense of e.g. \cite{Om}, that we will precise later. We anticipate this construction in order to point out here that in geenral, 
 	$\operatorname{Diff}(X)$ does \emph{not} have the subset diffeology 
 	from $C^\infty(X,X), $
 	 while it is the case when $X$ is a compact 
 	 boundaryless manifold,
 	  which is the most widely studied setting 
 	  in the actual literature.
 \end{rem}
Let us now assume that $X$ is a diffeological vector space with field of scalars a diffeological field $\mathbb{K}.$ 
\begin{Proposition}
	\begin{enumerate}
		\item The space $\mathcal{L}(X)$ of smooth linear maps from $X$ to $X$ is an algebra {with smooth addition, smooth multiplication and smooth scalar multiplication.}
		\item The group of invertible elements of $\mathcal{L}(X),$ that we note by $GL(X),$ is a diffeological subgroup of $\operatorname{Diff}(X).$ 
	\end{enumerate}
\end{Proposition}

\begin{rem} \label{hypo}
	Following the recent review \cite{Glo2022}, to which we refer for a complete exposition, on topological vector spaces that are not normed spaces, multilinear mappings are merely hypocontinuous, and not continuous. This is the case, when $E$ and $F$ are locally convex complete topological vector spaces, for the evaluation map $$\mathcal{L}(E,F) \times E \rightarrow F.$$
	Switching to the (subset) functional diffeology of $\mathcal{L}(E,F),$ where $E$ and $F$ are understood as diffeological vector spaces, the same evaluation map remains smooth. 
\end{rem}
	Following Iglesias-Zemmour, see \cite{Igdiff}, we do not assert that 
arbitrary diffeological groups have associated Lie algebras; however, 
the following holds, see \cite[Proposition 1.6.]{Les} and 
\cite[Proposition 2.20]{MR2016}.

\begin{Proposition} 
	Let $G$ be a diffeological group. Then the {internal} tangent cone at the neutral 
	element {}{${}^iT_eG$} is a diffeological vector space.
\end{Proposition}

\noindent
{
	The proof of Proposition \ref{Les}
	appearing in \cite{MR2016} uses explicitly the diffeologies
	$\mathcal{P}_1$ and  $\mathcal{P}_\infty$ which appear in 
	Proposition 3 and Remark 4 of this work.
}

\begin{Definition}
	The diffeological group $G$ is a \textbf{diffeological Lie group} 
	if and only if 
	the Adjoint action of $G$ on the diffeological vector space 
	$^iT_eG$ defines a {smooth} Lie bracket, {that is, noting by $c$ a path in $C^\infty(\R,G)$ such that $c(0)=e$ and $\partial_t c(0) = X,$ and noting by $Y$ an element of the internal tangent space ${}^iT_eG,$ 
		\begin{itemize}
			\item $\partial_t\left(\operatorname{Ad}_{c(t)}Y\right)|_{t=0} = [X,Y]\in {}^iT_eG$
			\item $(X,Y)\mapsto [X,Y]$ is smooth.
		\end{itemize}. }
	In this case, we call $^iT_eG$ the Lie algebra of $G$ and we denote 
	it by $\mathfrak{g}.$
\end{Definition}

The question of criteria for a diffeological group to be a diffeological Lie group has been addressed in \cite{Les} for diffeological groups, and in \cite{Lau2011} for Fr\"olicher groups. 

\begin{rem}
	Actually there is no proven example of diffeological group which is not a diffeological Lie group.
\end{rem}

The basic properties of adjoint, coadjoint actions, and of Lie brackets, remain {}{on diffeological Lie groups} globally the same
as in the case of finite-dimensional Lie groups, and the proofs are similar: see \cite{Les} 
and \cite{DN2007-1} for details.

	\subsection{On various tangent spaces}\label{s:varTX}

After our exposition of the construction of the internal tangent cone, we 
{describe} other existing {generalizations of the notion of tangent space} to a given diffeological space $X.$ Indeed, 
for a finite dimensional manifold, {tangent spaces} carry so many {technical} aspects {of interest} that the notion of tangent space can be generalized by many ways. Some of them appear as non equivalent, while there is actually a shared notion of differential form and of cotangent space of a diffeological space. Let us describe briefly 
{various} notions of tangent space for a diffeological space, and let us compare them. 

\subsubsection{The external tangent space}
Let us first adapt a classical algebraic construction to the diffeological category. The set of (diffeological) derivations on a diffeological space $X,$ noted by $\mathfrak{der}(X),$ is the space of smooth maps from $C^\infty(X,\R)$ to itself which satisfies the Leibnitz rule. The same way, the (diffeological) pointwise derivations on $X$ are the smooth maps $d$ from $C^\infty(X,\R)$ to $\R$ such that  $$ \exists  x \in X, \forall (f,g )\in (C^\infty(X,\R))^2, d(fg) = f(x) d(g) + g(x) d(f).$$
\begin{Definition}
	The external tangent space ${}^eTX$ is defined as 
	the set of diffeological pointwise derivations on $X.$
\end{Definition}
As a set of smooth maps, it can be endowed with the functional diffeology, and fixing the point $x \in X,$ one defines ${}^eT_xX \subset {}^eTX .$ 
\subsubsection{At the intersection of the internal and the external tangent spaces}\label{s:ieTX}

This {approach} 
is the one chosen in \cite{Ma2013,Ma2018-2,MR2016}, {and refined in \cite{GW2020}. This leads to two objects of different type: a cone \cite{Ma2013,Ma2018-2,MR2016} which can be completed into a vector space \cite{GW2020} following exactly the same procedure as the one applied in \cite{CW}. This last construction is not described here: we do not want exhaustivity in our presentation.}  

	The main idea behind the definition of this tangent space {in \cite{Ma2013,Ma2018-2,MR2016,GW2020}} is based on the definition of a tangent space on smooth manifolds through paths.
	More precisely, we consider the set of all paths that maps $ 0 $ to $ x $ and identify two of these paths with each other if their derivative coincide in $ 0 $.
	Therefore, an element in $ T_xX $ is  an {equivalence} 
	class $ [c] $ for $ c\colon  \mathbb{R} \rightarrow X $ such that $ c(0) = x $.
	The directional derivative of a function  $ f\colon X \rightarrow \mathbb{R} $ in direction $ v $ is given by the derivative $ \tfrac{\partial (f \circ c)(t) }{\partial t}_{\mid_{t = 0}} $ for $ c\colon \mathbb{R} \rightarrow X $ such that $ c(0) = x $ and $ c '(0) = v $. {This approach is then applied to an arbitrary diffeological space to produce a tangent cone the following way:} 


\begin{Definition}
	 For each $x\in X,$ we consider 
	$$
	C_{x}=\{c \in C^\infty(\R,X)| c(0) = x\}
	$$ 
	and we take the equivalence relation $\mathcal{R}$ given by 
	$$
	c\mathcal{R}c' \Leftrightarrow \forall f \in C^\infty(X,\R), \partial_t(f \circ c)|_{t = 0} = 
	\partial_t(f \circ c')|_{t = 0}.
	$$ 
	
\end{Definition}
	Equivalence classes of $\mathcal{R}$ are called {\bf germs} in \cite{Ma2013,Ma2018-2,MR2016} and are 
denoted by $\partial_t c(0)$ or $\partial_tc(t)|_{t=0}$. In the same references,  
the {\bf internal tangent cone} at $x$ is the quotient 
$C_x / \mathcal{R}.$ If 
$X = \partial_tc(t)|_{t=0} \in {}^iT_X, $ we define the pointwise derivation 
$f \mapsto Df(X) = \partial_t(f \circ c)|_{t = 0}\, .$
The reader may compare this definition to the one appearing in 
\cite{KM} for manifolds in the ``convenient"  $c^\infty-$setting. The internal tangent cone defined in section \ref{s:internal tangent} trivially differs from this one, which is also a subset of ${}^eT_xX.$
We note this {``internal-external'' tangent cone by ${C}^{ie}T_xX.$}
\begin{ex}\label{ex:cross}
	{Let $$X= \left\{(x,y) \in \R^2 \, | \, xy=0\right\},$$
	equipped with its subset diffeology. Then, following \cite{Ma2020-3}, $$C^{ie}T_0X = X$$ which shows that the internal-external tangent space is not in general a vector space.}
	\end{ex}
\begin{ex} \label{spagh2}
	Let us consider the example given in \cite{CW2022} of $\R^2$ equipped with the 1-dimensional diffology $\p_1(\R^2),$ called the spaghetti diffeology in the above reference. For  reasons similar to the ones given in example \ref{spagh}, the internal tangent space of $\R^2$ equipped with the diffeology $\p_1(\R^2)$ is infinite dimensional, with uncountable dimension, while as a consequence of Bohman's theorem, ${C}^{ie}T_x\R^2$ is 2-dimensional vector space) 
\end{ex}

	Following \cite{GW2020} we {introduce the following terminology:}
	\begin{Definition}
		\label{def:PathDerivative}
		Let $X$ be a diffeological space and $c \in C_x$. The {path derivative}  in direction $c$ is defined by 
		\begin{align*}
			\label{PathDerivative}
			\mathrm{d}_{c}\colon  C^{\infty}(X,\mathbb{R}) \rightarrow \mathbb{R},\, f \mapsto
			\mathrm{d}_{c}(f) := \dfrac{\partial}{\partial t}\left(f(c(t)\right)_{\mid_{t=0}}\in \mathbb{R}.
		\end{align*}
		One calls $\mathrm{d}_{c}(f) $ the \emph{path derivative} of the function $f\in C^{\infty}(X,\mathbb{R})$  in direction $c$.
	\end{Definition}
	
	{As explained in Example \ref{ex:cross}, see e.g. \cite{GW2020}, ${C}^{ie}T_xX$}
	does not provide a natural diffeological vector space structure. This motivates the next tangent space definition.
	\begin{Definition}
		\label{def:TxX}
		Let $X$ be a diffeological space and $x \in X$. 
		The {internal-external} tangent space {${}^{ie}T_xX$ is given by
		\begin{align*}
			{}^{ie}T_xX := \operatorname{span}{C}^{ie}T_xX.
		\end{align*}}
	\end{Definition}
	
	\begin{rem}
		{The set ${C}^{ie}T_xX$ is noted by $C_xX$ and is called the tangent cone  in \cite{GW2020}. Due to the multiple definitions of tangent objects in this work, we feel necessary to use the notation ${C}^{ie}T_xX$ because simplified notations woul be surely misleading.}
	\end{rem}
	
	Regarding the diffeologies on {${C}^{ie}T_xX$ and ${}^{ie}T_xX$} the following point of view is considered in \cite{GW2020} in order to obtain a \emph{diffeological tangent space} from the set {${C}^{ie}T_xX$:}
	\begin{itemize}
		\item 
		An element in {${}^{ie}T_xX$} is a finite sum of elements in {}{ $ {C}^{ie}T_xX $} that are multiplied by a scalar.
		\item
		{The diffeology  of ${C}^{ie}T_xX$ is the push-forward diffeology through the map $$c \in C_x \mapsto d_c \in {C}^{ie}T_xX$$ and a plot in ${}^{ie}T_xX$ is a mapping
		$$
		U\to {}^{ie}T_xX, \, u \mapsto \sum_{i=1}^{+\infty}\lambda_i(u)p_i(u)
		$$
		where $U\subset\mathbb{R}^n$ denotes an open subset with $n \in \mathbb{N}$,} $\lambda_i\colon  U \rightarrow\mathbb{R}$ is smooth {such that the sequence 
		$$\lambda = (\lambda_1, \cdots , \lambda_i , \cdots )$$
		is a plot of the subset diffeology of $$ \R^\infty = \bigcup_{k \in \N^*} \R^k \subset \R^\N$$}
		 and { each $p_i$ is a plot in ${C}^{ie}T_xX$ with domain $U$.}
	\end{itemize}
	The diffeology introduced on $T_xX$ is the so-called \emph{weak-diffeology} defined in \cite{vincent}. 
	Equipped with the weak diffeology, $ T_xX $ becomes the finest possible diffeological vector space such that, {with slight abuse of intuitionistic notations, $$ \p_{{C}^{ie}T_xX} \subset \p_{T_xX}.$$ }
}

\subsubsection{The diff-tangent space} \label{s:dTX}
We now extend the construction of \cite{Ma2020-3} on Fr\"olicher spaces to (general) diffeological spaces.

\begin{Definition} \label{def:dtangent}We use here the notations that we  used before for the definition of the internal tangent cone of section \ref{s:internal tangent}. 
	Let ${}^{d}T_xX$ be the subset of ${}^iT_xX$ defined by $$ {}^dT_xX = {}^dC_x /\mathcal{R}$$ with $${}^dC_x =\left\{ c \in C_x | \exists \gamma \in C^\infty(\R,\operatorname{Diff}(X)), c(.)=\gamma(.)(x) \hbox{ and } \gamma(0) = Id_X  \right\}$$
\end{Definition}

Through this definition, ${}^dT_xX$ is intrinsically linked with the tangent space at the identity ${}^iT_{Id_X}\operatorname{Diff}(X)$ described in \cite{Les} for any diffeological group (i.e. group equipped with a diffeology which makes composition and inversion smooth).
\begin{rem} \label{rq24}
	Let $\gamma \in C^\infty(\R,\operatorname{Diff}(X))$ such that $\gamma(0)(x)=x.$ Then $\lambda (x) = (\gamma(0))^{-1} \circ \gamma(.)(x)$ defines a smooth path $\lambda \in {}^dC_x.$ {{} Consequently,}  $${}^dC_x =\left\{ c \in C_x | \exists \gamma \in C^\infty(\R,\operatorname{Diff}(X)), c(.)=\gamma(.)(x) \hbox{ and } \gamma(0)= Id_X  \right\}$$ 
\end{rem}

	\subsection{Fiber bundles, principal bundles and vector bundles in diffeologies}

Let us now have a precise look at the notion of fiber bundle in classical (finite dimensional) fiber bundles. Fiber bundles, in the context of smooth finite dimensional manifolds, are defined by \begin{itemize}
	\item a smooth manifold  $E$ called total space
	\item a smooth manifold  $X$ called base space
	\item a smooth submersion $\pi: E \rightarrow X$ called fiber bundle projection
	\item a smooth manifold $F$ called typical fiber, because $\forall x \in X,  \pi^{-1}(x)$ is a smooth submanifold of $E$ diffeomorphic to $F.$
	\item a smooth atlas on $X,$ with domains $U \subset X$ such that $\pi^{-1}(U)$ is an open submanifold of $E$ diffeomorphic to $U \times F.$ We the get a system of local trivializations of the fiber bundle.
\end{itemize}
By the way, in order to be complete, a smooth fiber bundle should be the quadruple data $(E,X,F,\pi)$ (because the definition of $\pi$ and of $X$ enables to find systems of local trivializations). For short, this quadruple setting is often {{} denote}d by the projection map $\pi: E\rightarrow X.$

There exists some diffeological spaces which carry no atlas, so that, the condition of having a system of smooth trivializations in a generalization of the notion of fiber bundles is not a priori necessary, even if this condition, which is additional, enables interesting technical aspects \cite[pages 194-195]{MW2017}. So that, in a general setting, we do not need to assume the existence of local trivializations.
Now, following \cite{pervova}, in which the ideas from \cite[last section]{Sou} have been devoloped to vector spaces, the notion of quantum structure has been introduced in \cite{Sou} as a generalization of principal bundles, and the notion of vector pseudo-bundle in \cite{pervova}.The common idea consist in the description of fibered objects made of a total (diffeological) space $E,$ over a diffeological space $X$ and with a canonical smooth {{} bundle} projection $\pi: E \rightarrow X$ such as, $\forall x \in X,$ $\pi^{-1}(x)$ is endowed with a (smooth) algebraic structure, but for which we do not assume the existence of a system of local trivialization. 
\begin{enumerate}
\item For a diffeological vector pseudo-bundle, the fibers $\pi^{-1}(x)$ are assumed diffeological vector spaces, i.e. vector spaces where addition and multiplication over a diffeological field of scalars (e.g. $\R$ or $\mathbb{C}$) is smooth. We notice that \cite{pervova} only deals with finite dimensional vector spaces.
\item For a so-called ``structure quantique'' (i.e. ``quantum structure'') following the terminology of \cite{Sou}, a diffeological group $G$ is acting on the right, smoothly and freely on a diffeological space $E$. The space of orbits $X=E/G$ defines the base of the quantum structure $\pi: E \rightarrow X,$ which generalize the notion of principal bundle by not assuming the existence of local trivialization. In this picture, each fiber $\pi^{-1}(x)$ is isomorphic to $G.$
\end{enumerate}
From these two examples, we can generalize the picture. 
\begin{Definition}\label{pseu-fib}
Let $E$ and $X$ be two diffeological spaces and let $\pi:E\rightarrow X$ be a smooth surjective map. Then $(E,\pi,X)$ is a \textbf{diffeological fiber pseudo-bundle} if and only if  $\pi$ is a subduction. 
\end{Definition}
{{} Let us precise that we do not assume that there exists a typical fiber, in coherence with Pervova's diffeological vector pseudo-bundles. We
}
can give the following definitions:
\begin{Definition}
Let $\pi:E\rightarrow X$ be a diffeological fiber pseudo-bundle. Then:
\begin{enumerate}
	\item Let $\mathbb{K}$ be a diffeological field. $\pi:E\rightarrow X$ is a diffeological $\mathbb{K}-$vector pseudo-bundle if there exists: 
	\begin{itemize}
		\item a smooth fiberwise map $.\, :\mathbb{K} \times E \rightarrow E,$
		\item a smooth fiberwise map $+:E^{(2)} \rightarrow E$ where $$E^{(2)} = \coprod_{x \in X} \{(u,v) \in E^2\, | \, (u,v)\in \pi^{-1}(x)\}$$ equipped by the pull-back diffeology of the canonical map $E^{(2)} \rightarrow E^2,$
	\end{itemize}
	such that $\forall x \in X, $ $(\pi^{-1}(x),+,.)$ is a diffeological $\mathbb{K}-$vector bundle.
	\item $\pi:E\rightarrow X$ is a \textbf{diffeological gauge pseudo-bundle} if there exists \begin{itemize}
		\item a smooth fiberwise involutive map ${(.)}^{-1}\, E \rightarrow E,$
		\item a smooth fiberwise map $.\,:E^{(2)} \rightarrow E$ 
	\end{itemize}
	such that $\forall x \in X, $ $(\pi^{-1}(x),\, .\, )$ is a diffeological group with inverse map $(.)^{-1}.$
	\item  $\pi:E\rightarrow X$ is a \textbf{diffeological principal pseudo-bundle} if there exists a diffeological gauge pseudo-bundle $\pi': E' \rightarrow X$ such that, considering $$E\times_X E' = \coprod_{x \in X} \{(u,v)\, | \, (u,v)\in \pi^{-1}(x)\times \pi'^{-1}(x)\}$$ equipped by the pull-back diffeology of the canonical map $E\times_X E' \rightarrow E\times E',$ there exists a smooth map $E\times_X E' \rightarrow E$ which restricts fiberwise to a smooth free and transitive right-action $$\pi^{-1}(x) \times \pi'^{-1}(x) \rightarrow \pi^{-1}(x).$$ 
	\item $\pi:E\rightarrow X$ is a \textbf{Souriau quantum structure} if it is a diffeological principal pseudo-bundle with diffeological gauge (pseudo-)bundle $X\times G \rightarrow X.$ 
\end{enumerate}
\end{Definition}

{We now specialize to elementary classes of fiber bundles.}

\subsubsection{Tangent bundles}
The first motivating examples for these constructions were, of course, the tangent cones and the tangent spaces which are all fiber pseudo-bundles and sometimes vector pseudo-bundles.
{Indeed, analyzing the constructions in sections \ref{s:internal tangent}, \ref{s:ieTX} and \ref{s:dTX}, everything starts with a set of paths that we note here generically by $C_x(X)$ which generates the tangent cone at $x,$ that we note here generically by $CT_xX,$ and its diffeology through the quotient by an equivalence relation $\mathcal{R}_x.$ Let us now consider the \textbf{total} tangent space (in fact, tangent cone in general)
$$ CTX = \coprod_{x \in X} CT_xX,$$ for which basic intuition asserts that there exists a much better diffeology than the coproduct diffeology where fibers are arcwise disconnected one to each other. For the construction of this desired diffeology in the context of the internal tangent cone (section \ref{s:internal tangent}) and of the internal-external tangent cone (section \ref{s:ieTX}), one has only to remark that the evaluation maps at $t=0$ define a fiber pseudo-bundle projection
$$ \operatorname{ev}_0 : C^\infty(\R,X) \rightarrow X$$ which is a subduction, and such that $$\forall x \in X, \operatorname{ev}_0^{-1}(x)=C_x(X).$$ The fiberwise equivalence relations $\mathcal{R}_x$ define a global equivalence relation $\mathcal{R }$ on $C^\infty(\R,X)$ for which $$CTX = C^\infty(\R,X) / \mathcal{R},$$ which enables to defined a global quotient diffeology. } 
Therefore, one can define {}{from} the kinetic tangent cone at $x \in X$ the global {}{internal} tangent cone of $X$ and their diffeologies the following way: 

\begin{Definition}
	The  internal or kinetic tangent cone of $X$ as 
	$${}^iTX = \coprod_{x \in X} {}^iT_xX .$$
	The space ${}^iTX$ is endowed by the push-forward of the functional diffeology on $C^\infty(\R,X).$  
\end{Definition}
{The same definition holds for the definition of the total internal-external tangent cone $C^{ie}TX,$ by replacing the equivalence relations. Considering ${}^dTX,$ one has also to change in the definitions $C^\infty(\R,X)$ for its subset $$\coprod_{x \in X} {}^d C_x$$ which is endowed with its subset diffeology. }
\begin{Definition}
	Let $X$ be a Fr\"olicher space.
	we define, by $$ {}^dTX = \coprod_{x \in X} {}^dT_xX$$
	the diff-tangent bundle of $X.$ 
\end{Definition} 

By the way, we can get easily the following observations: 

\begin{Proposition} \label{prop:pdiff}
	Let $(X,\p)$ be a reflexive diffeological space, and let $\p_{\operatorname{Diff}}$ be the functional diffeology on $\operatorname{Diff}(X).$
	\begin{enumerate}
		\item \label{dt1} There exists a diffeology $\p(\operatorname{Diff}) \subset \p$ {{} which is} generated by the family of push-forward diffeologies : 
		$$ \left\{ (\operatorname{ev}_x)_{*}(\p_{\operatorname{Diff}}) \, | \, x \in X  \right\}.$$
		\item \label{dt2} $\forall x \in X, {}^dT_xX$ is the internal tangent cone of $(X,\p(\operatorname{Diff}))$ at $x.$
		\item \label{dt3} $\forall x \in X, {}^dT_xX$ is a diffeological vector space
		\item \label{dt4} The total diff-tangent space $${}^{d}TX = \coprod_{x \in X} {}^{d}TX \subset {}^{i}TX $$ is a vector pseudo-bundle for the subset diffeology inherited from ${}^{i}TX $ and also for the diffeology of internal tangent space of $(X,\p_{\operatorname{Diff}}).$
	\end{enumerate}
\end{Proposition}
\begin{proof}
	(\ref{dt1}) is a consequence of the definition of push-forward diffeologies the following way: the family $$\{ \p \hbox{ diffeology on } X \, | \, \forall x \in X, \, (\operatorname{ev}_x)_{*}(\p_{\operatorname{Diff}}) \subset \p\}$$ has a minimal element by Zorn Lemma. 
	
	(\ref{dt2}) follows from remark \ref{rq24}.
	
	(\ref{dt3}): The diffeology $\p(\operatorname{Diff})$ coincides with the diffeology made of plots which are locally of the form $\operatorname{ev}_x \circ p,$ where $x \in X$ and $p$ is a plot of the diffeology of $\operatorname{Diff}(X).$
	We have that $^{i}T_{Id}\operatorname{Diff}(X)$ is a diffeological vector space, following \cite{Les}. This relation follows from the differentiation of the multiplication of the group: given two paths $\gamma_1, \gamma_2$ in $C^\infty(\R,Diff(X)),$ with $\gamma_1(0)=\gamma_2(0)=Id,$ if $X_i = \partial_t\gamma_i(0)$ for $i \in \{1{{},}2\},$ then $$ X_1 + X_2 = \partial_t (\gamma_1 . \gamma_2) (0).$$ Reading locally plots in $\p(\operatorname{Diff}),$ we can consider only plots of the for $\operatorname{ev}_x \circ p,$ where $p$ is a plot in $\operatorname{Diff}(X)$ such that $p(0) = Id_X.$ By the way the vector space structure on $^{d}T_xX$ is inherited from  $^{i}T_{Id}\operatorname{Diff}(X)$ via evaluation maps.
	
	In order to finish to check (\ref{dt3}), we prove directly $(\ref{dt4})$ by describing its diffeology.
	
	For this, we consider $$C^\infty_0(\R, \operatorname{Diff}(X)) = \left\{ \gamma \in C^\infty_0(\R, \operatorname{Diff}(X)) \, | \, \gamma(0)=Id_X \right\}.$$
	Let $^{d}C = \coprod_{x \in X} ^{d}C_x.$ The total evaluation map \begin{eqnarray*}
		ev : & X \times C^\infty_0(\R, Diff(X)) \rightarrow & ^{d}C \\
		& (x,\gamma) \mapsto & ev_x \circ \gamma
	\end{eqnarray*}
	is fiberwise (over $X$), and onto. By the way we get a diffeology on $ ^{d}C$ which is the push-forward diffeology of $X \times C^\infty_0(\R, \operatorname{Diff}(X))$ by $\operatorname{ev}.$ Passing to the quotient, we get a diffeology on $^{d}TX$ which makes each fiber $^{d}T_xX$ a diffeological vector space trivially.
	
\end{proof}
{Therefore, the diff-tangent space is a vector pseudo-bundle, while one supplementary effort is necessary to generate a diffeology on ${}^{ie}TX$. For this, let us consider the map
$$\begin{array}{ccccc}l & : & C^\infty(X ,\R^\infty) \times C^{ie}TX^\N & \rightarrow &  {}^{ie}TX \\
&& (\lambda_1,\cdots, \lambda_k , \cdots ) \times (d_{c_1}, \cdots , d_{c_k}, \cdots ) & \mapsto & \sum_{k = 1}^{+ \infty} \lambda_k d_{c_k} 
\end{array}
$$
which produces the desired diffeology on ${}^{ie}TX$ by push-forward.

}
\subsubsection{Riemannian metrics on vector pseudo-bundles}
{\begin{Definition} \label{d:RMbundle}
		Let $\pi:E \rightarrow X$ be a real diffeological vector pseudo-bundle. A \textbf{Riemannian metric} on $E$ is a smooth map 
		$$ g : E^{(2)} \rightarrow \R$$ that is fiberwise on each $E_x$ a symmetric, definite and positive bilinear form. 
\end{Definition}
We have here to point out a first group of difficulties which are not present in the context of (classical) finite dimensional manifolds: 
\begin{itemize}
	\item The metric $g$ may not define an isomorphism between each fiber $E_x$ and its diffeological dual $\mathcal{L}(E_x,\R).$ 
	\item The fibers of $E$ may not be isomorphic. 
\end{itemize}
As a first class of examples, one can mention the vector pseudo-bundle ${}^dTX$ when $X$ is equipped with an internal Riemannian metric (see section \ref{s:rmdiff}), but one can wonder whether tangent cones can also carry Riemannian structures, whitout embedding in a larger Riemannian vector pseudo-bundle. To our actual state of knowledge, the only known such structure is described in section \ref{s:rmdiff}. }
\subsubsection{Diffeological principal bundles and connections} 
In \cite[Article 8.32]{Igdiff} Iglesias-Zemmour gives a
definition of a connection on a principal $G$-bundle in terms of
paths on the total space $P,$ generalising the classical notion 
of path lifting for principal bundles with finite-dimensional
Lie groups as structure groups.

\begin{Definition}\label{d:iz-connection}
	Let $G$ be a diffeological group, and let
	$\pi\colon P\to X$ be a principal $G$-bundle.
	Denote by $\pathloc(P)$ the diffeological space of
	\textbf{local paths} (see \cite[Article 1.63]{Igdiff}), and by 
	$\operatorname{tpath}(P)$ the \textbf{tautological bundle of local paths} 
	$$\operatorname{tpath}(P):=\{(\gamma,t)\in\pathloc(P)\times\R\mid t\in D(\gamma)\}.$$ 
	A \textbf{diffeological connection} is a smooth map 
	$H \colon \operatorname{tpath}(P)\to\pathloc(P)$ satisfying the following properties
	for any $(\gamma,t_0)\in \operatorname{tpath}(P)$:
	\begin{enumerate}
		\item the domain of $\gamma$ equals the domain of $H(\gamma,t_0)$,
		\item $\pi\circ\gamma=\pi\circ H(\gamma,t_0)$,
		\item $H(\gamma,t_0)(t_0)=\gamma(t_0)$,
		\item $H(\gamma\cdot g,t_0)=H(\gamma,t_0)\cdot g$ for all $g\in G$,
		\item $H(\gamma\circ f,s)=H(\gamma,f(s))\circ f$ for any smooth map
		$f$ from an open subset of $\R$ into $D(\gamma)$,
		\item $H(H(\gamma,t_0),t_0)=H(\gamma,t_0)$.
	\end{enumerate}
\end{Definition}

Another formulation of this definition can be found in \cite{Ma2013} under the terminology 
of \textbf{path-lifting}.

\begin{rem}\label{r:iz-connection}
	Diffeological connections satisfy many of the usual properties that 
	classical connections on a principal $G$-bundle (where $G$ is a finite-
	dimensional Lie group) enjoy; in particular, they admit unique horizontal 
	lifts of paths in $\pathloc(M)$ 
	\cite[Article 8.32]{Igdiff}, and they pull back by smooth maps 
	\cite[Article 8.33]{Igdiff}.
\end{rem}

\begin{Proposition}
	Let $V$ be a vector space. Then, $G$ acts smoothly from the right on 
	the space $\Omega(P,V)$ of $V$-valued differential forms on $P$ by setting
	$$ 
	\forall (g,\alpha) \in  \Omega^n(P,V) \times G, \forall p\in \p(P), \quad 
	(g_*\alpha)_{g.p} = \alpha_p \circ (dg^{-1})^n \; .
	$$
\end{Proposition}
\begin{proof}
	$G$ acts smoothly on $P$ so that, if $p \in \p(P),$ $g.p \in \p(P)$.
	The right action is now well-defined, and smoothness is trivial.
\end{proof}

\begin{Definition}
	Let $\alpha \in \Omega(P;\mathfrak{g}).$ The differential form $\alpha$ is 
	\textbf{right-invariant} if and only if, for each $p \in \p(P),$ and for each 
	$g \in G,$
	$$\alpha_{g.p} = Ad_{g^{-1}} \circ g_*\alpha_p \; .$$
\end{Definition}

Now, let us turn to connections and holonomy. Let $p \in P$ and
let $\gamma$ be a smooth path in $P$ starting at $p.$

\begin{Definition}
	A \textbf{connection} on $P$ is a $\mathfrak{g}-$valued right-invariant 
	$1$-form $\theta,$ , such that, for each $ v \in \mathfrak{g},$ for any path 
	$c : \R \rightarrow G$ such that 
	$$
	\left\{\begin{array}{ccr} c(0)& = & e_G\\ \partial_tc(t)|_{t=0}&= & v \; \; ,\end{array} \right.
	$$ 
	and for each $p \in P$ we have: 
	$$\theta(\partial_t(p.c(t))_{t = 0})=v \; .$$
\end{Definition}
Now we assume that $dim(M)\geq 2$ and we fix a connection
$\theta$ on $P.$

\begin{Definition}
	Let $\alpha \in \Omega(P;\mathfrak{g})$ be a $G-$invariant $1$-form. Let 
	$\nabla \alpha = d\alpha - {\frac{1}{ 2}}[\theta,\alpha]$ be the horizontal derivative 
	of $\alpha.$ The curvature $2$-form induced by $\theta$ is 
	$$ \Omega = \nabla \theta \; .$$ 
\end{Definition}

\section{Diffeologies in functional equations}\label{diffeq}
	\subsection{Principal bundles, fully regular Lie groups and holonomy} \label{s:reg}
	\subsubsection{Motions on groups and (partial) differential equation}
	{A wide class of first order differential equations are usually solved globally by the use of the  integration of smooth paths $v$ in a Lie algebra into smooth paths $g$ in a Lie group, satisfying a differential equation of the type \begin{equation}\label{eq:logder}g(t)^{-1}\partial_tg(t) = v(t). \end{equation}
	When $v$ is constant, $g(t)=\operatorname{exp}(tv)$ produces, when it exists, the only solution $g$ such that $g(0)=e.$
	This approach is the initial motivation for the development of the notion of Lie grops by Sophus Lie. For finite dimensional examples, this approach gives a geometric way to solve differential equations with value in Lie groups, but these techniques can be also of interest for infinite dimensional objects.
	
	More precisely, the first examples of interest, where the exponential map is not easy to define, are linear differential equations of the form 
\begin{equation} \label{eq:linPDE} \frac{df}{dt} = A(f) \end{equation}
where $A$ is a differential operator of positive order acting in a space of smooth functions $F.$ 
\begin{ex}
	When $F = C^\infty(S^1,\R)$, consider $A = - \frac{d^2}{dx^2}.$ Then there is an operator $\operatorname{exp}(tA)$ which is unbounded, that solves (\ref{eq:linPDE} with eigenvalues $e^{tn^2}.$ 
	
	Modifying $A$ into $\frac{d^2}{dx^2}$, the solution $\operatorname{exp}(tA)$ is a family of Fredholm operators of index $0$ with eigenvalues $e^{-tn^2} \leq 1.$  
\end{ex} 
From this elementary example, which can be expanded and complexified as far as wanted, one can easily understand that control of the regularity of the solutions for these equations, and the methods for approximating them numerically, highly depends on the nature of the operator $A.$ The same way, equations of the form 
\begin{equation} \label{eq:Liebracket}
\frac{dS}{dt} = \left[A(t), S(t)\right],
\end{equation}
where $S$ is an operator-valued solution in a Lie algebra $\mathfrak{g}$ and $A$ is a smooth path in $\mathfrak{g},$
have a formal solution $$S(t) = \operatorname{Ad}_{G(t)} S(0)$$ where $G$ is the solution of equation (\ref{eq:logder})  . 

These elementary facts
can be generalized in two ways that we wish to mention here: 
\begin{itemize}
	\item Symmetries of a (partial) differential equation generate transformation groups which transform one solution to another. In practice, one often deals with \textit{infintesimal symmetries} which are formally tangent vectors to the group of symmetries \cite{O}. Integrating a Lie algebra of infinitesimal symmetries is less easy and requires technical abilities in the field of infinite dimensional Lie groups. One can find an example of such concerns for symmetries of the 3d-Euler equation in \cite{Rob}. These indications help to solve the PDE by reduction, in general case by case and when it produces ``enough'' symmetries, but anyway the symmetries of PDEs are in general of interest for the global understanding of the space of solutions.
	\item Integrable systems are differential equations that are equivalent to a zero curvature condition, that is, they can be expressed as $$dS \pm [S,S] =0$$ where $S$ is a well-chosen 1-form built from the initial PDE. The main clue for solving the problem is that this precise equation is a so-called \textit{zero curvature equation}, that integrates globally when one deals with on a trivial finite dimensional principal bundle in which $S$ is understood as a connexion 1-form.   
\end{itemize} 
 In section \ref{diffeq}, we expose what is the actual state-of-the-art of the application of diffeologies to the necessary technical properies on diffeological Lie groups, in view of potential applications to PDEs that can be expressed in an infinite dimensonal Lie group. }
	\subsubsection{On regular diffeological Lie groups}
	Since we are interested in infinite-dimensional analogues of Lie 
groups, we need to consider tangent spaces of 
diffeological spaces, and we have to deal with Lie algebras and 
exponential maps. 
We state, after \cite{Les,DN2007-1,CW} the 
following definition:

\begin{Definition} \label{reg1} \cite{Les} 
	A {diffeological} Lie group $G$ with Lie algebra $\mathfrak{g}$
	is called \textbf{regular} if and only if there is a smooth map 
	\[
	\operatorname{Exp}:C^{\infty}([0;1],\mathfrak{g})\rightarrow C^{\infty}([0,1],G)
	\]
	such that $g(t)=\operatorname{Exp}(v(t))$ is the unique solution
	of the differential equation \begin{equation}
		\label{loga}
		\left\{ \begin{array}{l}
			g(0)=e\\
			\frac{dg(t)}{dt}g(t)^{-1}=v(t)\end{array}\right.\end{equation}
	We define the exponential function as follows:
	\begin{eqnarray*}
		exp:\mathfrak{g} & \rightarrow & G\\
		v & \mapsto & exp(v)=g(1) \; ,
	\end{eqnarray*}
	where $g$ is the image by $\operatorname{Exp}$ of the constant path $v.$ 
\end{Definition}
{\begin{rem}
	Equation (\ref{loga}) is called \textbf{right logarithmic} equation while Equation \ref{eq:logder} is called \textbf{left logarithmic}. The correspondence between the two equations is actually well-exposed in \cite{KM}, and the existence of the solution of the first one is equivalent to the existence of a solution to the second one.
\end{rem}}
When the Lie group $G$ is a vector space $V$, the notion of regular Lie group specialize to what is called 
{\em regular vector space} in \cite{Ma2013} and {\em integral vector space} in \cite{Les}; we follow the first
terminology.

\begin{Definition} \label{reg2} \cite{Les}
	Let $(V,\p)$ be a {diffeological} vector space. 
	The space $(V,\p)$ is \textbf{integral} or \textbf{regular} if there is a smooth map
	$$ 
	\int_0^{(.)} : C^\infty([0;1];V) \rightarrow C^\infty([0;1],V)
	$$ 
	such that $\int_0^{(.)}v = u$ if and only if $u$ is the unique solution of the differential equation
	\[
	\left\{ \begin{array}{l}
		u(0)=0\\
		u'(t)=v(t)\end{array}\right. .\]
\end{Definition}

This definition applies, for instance, if $V$ is a complete locally convex topological vector space equipped 
with its natural Fr\"olicher structure given by the Fr\"olicher completion of its n\'ebuleuse diffeology, see
\cite{Igdiff,Ma2006-3,Ma2013}. We give now the corresponding notion for derivatives, after remarking that $\forall v \in V, V \subset {}^iT_vV$ through the identification of linear paths $t \mapsto tv$ with $v:$ 

\begin{Definition}
	Let $V$ be a diffeological vector space. Then $V$ is {\bf co-regular}
 if $$ \forall v \in V, {}^iT_vV = V.$$\end{Definition}

This property, which seems natural, highly depends on the diffeology considered, see Example \ref{spagh2}. This may explain why all the authors since \cite{Les} and till now, and even us, have considered this problem as minor or even negligible. 

\begin{Definition}
	Let $G$ be a {diffeological} Lie group with Lie algebra $\mathfrak{g}$. 
	Then, $G$ is { fully regular, i.e. regular} with integral Lie algebra if  
	$\mathfrak{g}$ is integral and $G$ is regular in the sense of 
	Definitions $\ref{reg1}$ and $\ref{reg2}$.
\end{Definition}

We finish this section with two structural results essentially proven in 
\cite{Ma2013} in a restricted setting for the second one. 

\begin{Theorem}\label{exactsequence} 
	Let
	$$ 
	1 \longrightarrow K \stackrel{i}{\longrightarrow} G \stackrel{p}{\longrightarrow}  H \longrightarrow 1 
	$$
	be an exact sequence of Fr\"olicher Lie groups, such that there is 
	a smooth section $s : H \rightarrow G,$ and such that the trace 
	diffeology  on $i(K) \subseteq G$ coincides with the push-forward 
	diffeology from $K$ to $i(K).$ We consider also the corresponding 
	sequence of Lie algebras
	$$ 
	0 \longrightarrow \mathfrak{k} \stackrel{i'}{\longrightarrow} 
	\mathfrak{g} \stackrel{p}{\longrightarrow}  
	\mathfrak{h} \longrightarrow 0 \; . 
	$$
	Then,
	\begin{itemize}
		\item The Lie algebras $\mathfrak{k}$ and $\mathfrak{h}$ are 
		integral if and only if the Lie algebra $\mathfrak{g}$ is integral;
		\item The Fr\"olicher Lie groups $K$ and $H$ are regular if and 
		only if the Fr\"olicher Lie group $G$ is regular.
	\end{itemize}
\end{Theorem}

\begin{Theorem} \label{regulardeformation} 
	Let $(A_n)_{n \in \mathbb{N}} $ be a sequence of coregular and integral   
	Fr\"olicher vector spaces equipped with a graded smooth 
	multiplication operation on $\bigoplus_{n \in \mathbb{N}^*} A_n\, ,$ 
	i.e. a multiplication such that for each $n,m \in \mathbb{N}^*$, 
	$A_n .A_m \subset A_{n+m}$ is smooth with respect to the 
	corresponding Fr\"olicher structures.  
	
	\begin{itemize}
		\item 
		Let us define the (non unital) algebra of formal series:
		$$
		\mathcal{A}= \left\{ \sum_{n \in \mathbb{N}^*} a_n | 
		\forall n \in \mathbb{N}^* , a_n \in A_n \right\}\; ,
		$$
		equipped with the Fr\"olicher structure of the infinite product.
		Then, the space 
		$$1 + \mathcal{A} = 
		\left\{ 1 + \sum_{n \in \mathbb{N}^*} a_n | 
		\forall n \in \mathbb{N}^* , a_n \in A_n \right\} $$
		is a regular Fr\"olicher Lie group with 
		integral Fr\"olicher Lie algebra $\mathcal{A}.$
		Moreover, the exponential map defines a smooth bijection 
		$\mathcal{A} \rightarrow 1+\mathcal{A}.$
	\end{itemize}
\end{Theorem}

A result similar to Theorem \ref{exactsequence} is also 
valid for Fr\'echet Lie groups, see \cite{KM}.
\subsubsection{On the holonomy of a connection in a diffeological principal bundle}
Now, let $P$ be a principal bundle. Let $p \in P$ and
$\gamma$ a smooth path in $P$ starting at $p,$ defined on $[0,1].$
Let $H_\theta \gamma (t) = \gamma(t)g(t)$, where $g(t) \in C^\infty([0,1];\mathfrak{g})$ 
is a path satisfying the differential equation:
$$
\left\{ \begin{array}{c} \theta \left( \partial_t H_\theta\gamma(t) \right) = 0  \\ 
	H_\theta\gamma(0)=\gamma(0) \end{array} \right.
$$
The first line of this equation is equivalent to the differential
equation $$g^{-1}(t)\partial_tg(t) = -\theta(\partial_t\gamma  (t))$$ which is
integrable, and the second line is equivalent to the initial condition $g(0)=e_G.$
This shows that horizontal lifts are well-defined, as in the standard case
of finite-dimensional manifolds. Moreover, the map $H_\theta(.)$ defines trivially a
diffeological connection. This enables us to consider the holonomy
group of the connection. Notice that a straightforward adaptation
of the arguments of \cite{Ma2013} shows that the holonomy group does not depend (up to 
conjugation and up to the choice of connected component of $M$) on the choice of the base 
point $p.$
This definition allows us to consider reductions of the structure group. {}{Following \cite{Ma2013},}

\begin{Theorem} \label{Courbure} 
	We assume that $G_1$ and $G$ are regular Fr\"olicher groups
	with regular Lie algebras $\mathfrak{g}_1$ and $\mathfrak{g}.$
	Let $\rho: G_1 \mapsto G$ be an injective morphism of Lie groups.
	If there exists a connection $\theta$ on $P$, with curvature $\Omega$, 
	such that for any smooth $1$-parameter family $H_\theta c_t$ of 
	horizontal paths starting at $p$, and for any smooth vector fields $X,Y$ 
	in $M$, the map
	\begin{eqnarray} 
		s, t \in [0,1]^2 & \rightarrow & \Omega_{Hc_t(s)}(X,Y)  \label{g1}
	\end{eqnarray}
	is a smooth $\mathfrak g_1$-valued map (for the $\mathfrak g _1 -$ 
	diffeology),
	\noindent
	and if $M$ is simply connected, then the structure group $G$ of $P$ reduces to 
	$G_1,$ and the connection $\theta$ also reduces.
\end{Theorem}

We can now state the announced Ambrose-Singer theorem, using the terminology of \cite{Rob} 
for the classification of groups via properties of the exponential map \cite{Ma2013}:

\begin{Theorem}
	\label{Ambrose-Singer} 
	Let $P$ be a principal bundle whose  structure group is a fully regular Fr\"olicher 
	Lie group $G$. Let $\theta$ be a connection on $P$ and $H_\theta$ the associated 
	diffeological connection.
	\begin{enumerate}
		\item For each $p \in P,$ the holonomy group $\Hol_p^L$ is a
		diffeological subgroup of $G$, which does not depend on the choice of
		$p$ up to conjugation.
		
		\item There exists a second holonomy group $H^{red},$ $\Hol \subset H^{red},$
		which is the smallest structure group for which there is a subbundle $P'$ to
		which $\theta$ reduces. Its Lie algebra is spanned by the curvature elements, i.e.
		it is the smallest integrable Lie algebra which contains the 
		curvature elements.
		
		\item If $G$ is a Lie group (in the classical sense) of type I or II,
		there is a (minimal) closed Lie subgroup $\bar{H}^{red}$ (in the 
		classical sense) such that $H^{red}\subset \bar{H}^{red},$
		whose Lie algebra is the closure in $\mathfrak{g}$ of the Lie algebra 
		of $H^{red}.$ $\bar{H}^{red}$
		is the smallest closed Lie subgroup of $G$ among the structure groups
		of closed sub-bundles $\bar{P}'$ of $P$ to which $\theta$ reduces.
	\end{enumerate}
\end{Theorem}

From \cite{Ma2013} again, we have the following result:

\begin{Proposition} \label{0-courbure}
	If the connection $\theta$ is flat and $M$ is connected and simply 
	connected, then for any path $\gamma$
	starting at $p \in P,$ the map $$\gamma \mapsto H_\theta\gamma(1)$$ 
	depends only on $\pi(\gamma(1))\in M$, and it defines
	a global smooth section $M \rightarrow P.$ Therefore, $P = M \times G.$
\end{Proposition}

Let us precise a little bit more this result (see \cite[section 40.2]{KM} 
for an analogous statement in the $c^\infty$-setting):

\begin{Theorem} \label{Hslice}
	Let $(G,\mathfrak{g})$ be a regular Lie group with regular Lie algebra 
	and let $X$ be a simply connected
	Fr\"olicher space. Let $\alpha \in \Omega^1(M,\mathfrak{g})$ such that 
	\begin{equation} \label{beta1}
		d\alpha + [\alpha,\alpha]=0\; .
	\end{equation}
	Then there exists a smooth map $$f : X \rightarrow G $$
	such that $$df.f^{-1} = \alpha.$$ Moreover, we move from one solution 
	$f$ to another by applying the
	Adjoint action of $G$, pointwise in $x \in X$.
\end{Theorem}

We remark that the theorem also holds if we consider the equation
$$d\alpha - [\alpha,\alpha]=0$$ instead of (\ref{beta1}); we only need to 
change left logarithmic derivatives for right logarithmic derivatives,
and Adjoint action for Coadjoint action. The correspondence
between solutions is given by the inverse map $f \mapsto f^{-1}$
on the group $C^\infty(X,G).$

\subsection{Optimization on diffeological spaces}
\label{sec:Optimization}

In general, optimization methods aim at the minimization or maximization of an objective functional. Often, this functional depends on the solution of a partial differential equation (PDE). Examples of a PDE are the compliance of an elastic structure or a dissipated energy in a viscous flow. 
Since a maximization problem can be expressed as an minimization problem by considering the negative objective functional, often, only minimization problems are considered in the literature.
In general, optimization methods are iterative methods that generate updates such that the objective functional is reduced (in the case of a minimization problem). 
We concentrate first on shape optimization to motivate the consideration of diffeological spaces in optimization techniques. Afterwards, we summarize the recent findings regarding optimization techniques on diffeological spaces.


\subsubsection{A motivation for considering diffeological spaces in (shape) optimization}

Shape optimization has many applications and a large variety of methods. Application examples are acoustic shape optimization {\cite{Schmidt2016}}, optimization of interfaces in transmission problems {\cite{Gangl2015,Paganini15}}, electrochemical machining {\cite{Hintermueller2011}}, image restoration and segmentation {\cite{Hintermueller2004}} and inverse modelling of skin structures {\cite{Naegel2015}}.

Shape optimization problems are usually solved by using iterative methods such that one starts with an initial shape and then gradually evolves it be morphing it into the optimal shape.
In order to formulate optimization methods and enable their theoretical investigations, one needs to define first what we describe as a shape. There are multiple options, e.g. the usage of landmark vectors~{\cite{Cootes1995,Hafner2000,Kendall1984,Perperidis2005,Soehn2005}, plane curves \cite{MichMum,Michor2007,Michor2007a,Mio2006} or surfaces in higher dimensions \cite{Bauer2011a,Bauer2012,Kilian2007,Kurtek2010,Michor2005}, boundary contours of objects \cite{Fuchs2009,Ling2007,Rumpf2009}, multiphase objects \cite{Wirth2010}, characteristic functions of measurable sets \cite{Zolesio2007} and morphologies of images \cite{Droske2007}.}
In general, a shape space does not have a vector space structure. Instead, the next-best option is to aim for a manifold structure with an associated Riemannian metric. 
In case a manifold structure cannot be established for the shape space in question, an alternative option is a diffeological space structure.
In contrast to shape spaces as  Riemannian manifolds, research for diffeological spaces as shape spaces has just begun, see e.g.~{\cite{KW,KW21}. }

In the following, we summarize the main findings of  {\cite{KW21},} which suggests the use of diffeological spaces instead of smooth manifolds in the context of PDE constrained shape optimization.

In shape optimization, one investigates shape functionals.
For a suitable shape space $\mathcal{U}$ a shape functional is a function $J\colon \mathcal{U}\to\R$.
An unconstrained shape optimization problem is  given by
\begin{equation}
\label{OptProb}
\min_{\Omega\in \mathcal{U}} J(\Omega).
\end{equation}
In applications, shape optimization problems are often constrained by equations like PDEs.

Regarding suitable shape spaces in optimization problems, {\cite{KW21}} starts by concentrating on one-dimensional smooth shapes.
A one-dimensional smooth shape is a $C^{\infty}$-boundary of a simply connected and compact set $\Omega$. 
These shapes can be interpreted as  smooth single-closed curves represented by embeddings from $S^1$ into $\mathbb{R}^2$.
Since we are only interested in the shape itself, i.e., the image of the curve, we are not interested in re-parametrisations.
Therefore one can consider the shape space given in \cite{MichMum}:
$$
B_e(S^1,\mathbb{R}^2) := \mathrm{Emb}(S^1,\mathbb{R}^2) / \mathrm{\operatorname{Diff}}(S^1),
$$
where $\mathrm{Emb}(S^1,\mathbb{R}^2)$ is the set of all embeddings from $S^1$ to $\mathbb{R}^2$ and $\mathrm{\operatorname{Diff}}(S^1)$ the space of all diffeomorphisms from $S^1$ into itself.
The space $B_e(S^1,\mathbb{R}^2)$ is a smooth manifold, see \cite{KM}.
The tangent space is given by
$$
T_cB_e(S^1,\mathbb{R}^2) \simeq \{ h \mid h= \alpha n, \alpha \in C^{\infty}(S^1) \},
$$
where $n$ denotes the exterior unit normal field to the shape boundary $c: S^1 \to \mathbb{R}^2$.
The space of one-dimensional smooth shapes is a Riemannian manifold for various Riemannian metrics like the almost local metrics and Sobolev metrics; see e.g. {\cite{MichMum,Bauer2011a,Bauer2012,Michor2005, Michor2007}. }
All these metrics arise from the standard $L^2$-metric by putting weights (almost local metrics), derivatives (Sobolev metrics) or both (weighted Sobolev metrics) in it. 

{In \cite{KW21},} for $\mathcal{U}$ the shape space $B_e(S^1,\R^2)$ combined with the first Sobolev metric and the Steklov-Poincaré metric {(cf. \cite{Schulz2016})} is considered.

\begin{rem}
	One can generalize the shape space $B_e(S^1,\mathbb{R}^2)$ and its results for a compact manifold $M$ and a Riemannian manifold $N$ with $\operatorname{dim}(M) < \dim(N)$ {(cf., e.g., \cite{Michor2005}).}
\end{rem}

In general, to formulate optimization methods on a Riemannian shape space $(\mathcal{U},g)$, a Riemannian shape gradient with respect to $g$  is needed.
A connection of the shape space $B_e(S^1,\R^2)$ together with the first Sobolev metric and the Steklov-Poncaré metric to shape calculus is given in {\cite{KW21}.} Moreover, both approaches are compared  to each other {(cf. \cite{VSMS15,KW21}).} In particular, it should be mentioned that working with the Steklov-Poincaré metric has several computational advantages regarding the finite element mesh, which needs to be considered to solve PDE constraints. We refer the reader to  {\cite{VSMS15,Siebenborn2016,Schulz2016}.}

\begin{rem}
	\label{RemOD}
It is well known that shape optimization algorithms may invoke substantial transformations of the initial shape in the course of the optimization process, which often leads to a deterioration of the cell aspect ratios of the underlying computational meshes. 
Usually, optimization problems in function spaces can be solved using one of two different approaches: discretize--then--optimize (DO) and optimize--then--discretize (OD). 
Shape optimization problems are no exception.
We refer the interested reader to { \cite{BLWH}} for the description of differences between the OD and DO approaches by a prototypical example of a PDE-constrained shape optimization problem.
\end{rem}

Regardless the use and benefits of these OD-approaches (for which we refer to the literature), the shape space $B_e(S^1,\R^2)$ itself limits the application of the methods since it only contains smooth shapes.
From a numerical point of view it is desirable to weak the smoothness assumption of the shapes. Numerical experiments show that also shapes with kinks can be handled with the Steklov-Poincaré metric; see { \cite{VSMS15, Schulz2016, Siebenborn2016}.}
Thus, another shape space definition is necessary.
In { \cite{KW21},} the space of so called $H^{1/2}$-shapes is defined:
\begin{Definition}
	Let $\Gamma_0 \subset \mathbb{R}^d$ be a $d$-dimensional Lipschitz shape (the boundary of a non-trivial Lipschitz domain). 
	The space of all $d$-dimensional $H^{1 / 2}$-shapes is given by
	$$
	\mathcal{B}^{1 / 2}\left(\Gamma_0, \mathbb{R}^d\right):=\mathcal{H}^{1 / 2}\left(\Gamma_0, \mathbb{R}^d\right) / \sim,
	$$
	where
	$$
	\begin{aligned}
		& \mathcal{H}^{1 / 2}\left(\Gamma_0, \mathbb{R}^d\right) \\
		& :=\left\{w: w \in H^{1 / 2}\left(\Gamma_0, \mathbb{R}^d\right) \text { injective, continuous; } w\left(\Gamma_0\right) \text { Lipschitz shape }\right\}
	\end{aligned}
	$$
	and the equivalence relation $\sim$ is given by
	$$
	w_1 \sim w_2 \Leftrightarrow w_1\left(\Gamma_0\right)=w_2\left(\Gamma_0\right) \text {, where } w_1, w_2 \in \mathcal{H}^{1 / 2}\left(\Gamma_0, \mathbb{R}^d\right) \text {. }
	$$
\end{Definition}
Since the space $\mathcal B^{1/2}$ is a challenging one, it is so far unclear if it is a Riemannian manifold or even a manifold.
Therefore, { \cite{KW21}} came up with the idea of using a different point of view and drop the restrictiveness of Riemannian manifolds and starts to consider the opportunities within diffeological spaces. 
Due to the wide variety of natural diffeological spaces, the $\mathcal B^{1/2}$ shape space is a diffeological space {(cf. \cite{KW21}).}
In addition, $\mathcal B^{1/2}$ is way less restrictive than $B_e$.
However, diffeological spaces are yet to be established in the area of optimization.
The first results regarding optimization techniques on diffeological spaces are summarized in the next subsection

\subsubsection{General concepts for optimization on diffeological spaces}
	In the following, we give a brief summary of \cite{GW2020}, in which optimization approaches on diffeological spaces are formulated.
	In order to formulate optimization methods on diffeological spaces, the 
	steepest decent method on manifolds is considered extended to diffeological spaces.
	For convenience, we formulate the steepest descent method on a complete Riemannian manifold~$(M,g)$ in algorithm \ref{Algo}.
	
	\begin{algorithm}
		\caption{Steepest descent method on a complete Riemannian manifold~$(M,g)$}
		\label{Algo}
		\begin{algorithmic}
			\vspace{.3cm}
			\STATE{ \textbf{Require:} Objective function $f$ on $M$; Levi-Civita connection $\nabla$ on $(M,g)$; 
				\\\phantom{\textbf{Require:} }step size strategy.}
			\vspace{.1cm}
			\STATE{ \textbf{Goal:} Find the solution of $\min\limits_{x\in M}f(x)$.}
			\vspace{.1cm}
			\STATE{ \textbf{Input:} Initial data $x_0\in M$. }
			\vspace{.3cm}
			\STATE{ \textbf{for} $k=0,1,\dots$ \textbf{do}}
			\vspace{.1cm}
			\STATE{ [1] Compute $\text{grad}f(x_k)$ denoting the Riemannian shape gradient of $f$ in $x_k$.}
			\vspace{.1cm}
			\STATE{ [2] Compute step size $t_k$.}
			\vspace{.1cm}
			\STATE{ [3] Set $			x_{k+1}:= \operatorname{exp}_{x_k}\left(-t_k \text{grad}f(x_k)\right)$, where $\mathrm{exp}\colon TM\to M$ denotes the exponential map} 
			\vspace{.1cm}
			\STATE{ \textbf{end for}}
			\vspace{.3cm}
		\end{algorithmic}
	\end{algorithm}
	
	In \cite{GW2020}, diffeological counterparts to the known and needed objects for the steepest decent method on smooth manifolds are considered,  i.e., diffeological variants of a Riemannian manifold, a Riemannian gradient, a Levi-Civita connection, as well as a retraction. In the following, we introduce these objects from \cite{GW2020}.
	
	\begin{rem}
		The article	\cite{GW2020} does not consider a diffeological version of the exponential map, since for optimization purpose the exponential map is usually replaced by a retraction that approximates the exponential map. 
		Furthermore, the existence of an exponential map in a diffeological sense is totally unclear.
	\end{rem}

	\subsubsection{Diffeological gradient}
		For optimization purposes the existence of a gradient is more or less indispensable, since it provides a direction of decent under special assumptions.
		In order to define a diffeological version of the gradient we first define a diffeological Riemannian space, which results in the definition of a diffeological gradient. {We copy here the definition given in \cite{GW2020}, that we compare with the previous definitions.}

		\begin{Definition}[Diffeological Riemannian space]
			Let $X$ be a diffeological space. We say $X$ is a {diffeological Riemannian space} if there exists a smooth map 
			\begin{align*}
				\mathrm{Sym}({{}^{ie}TX},\mathbb{R}),\, x \mapsto g_x
			\end{align*}
			such that
			$
			g_x\colon { {}^{ie}T_xX \times {}^{ie}T_xX} \to \mathbb{R}
			$
			is smooth, symmetric and positive definite.
			Then, we call the map $g$ a diffeological Riemannian metric.
		\end{Definition}
	
	{\begin{rem}
				This definition fits with the defninition of a Riemannian metric along the lines of Definition \ref{d:RMbundle} on the diffeological vector bundle ${}^{ie}TX$ but differs from Definition \ref{d:RM} of an internal Riemannian metric.
			\end{rem}}
		
		\begin{Definition}[Diffeological gradient]
			Let $X$ be a diffeological Riemannian space. The diffeological gradient $ \operatorname{grad}f $ of a function $f \in C^{\infty}(X)$ in $x \in X$ is defined as the solution of 
			\begin{align*}
				g_x(\operatorname{grad} f, \mathrm{d}_{c}) = \mathrm{d}_{c}(f).
			\end{align*}
		\end{Definition}
		
		\begin{rem}
			The existence of a diffeological gradient for a diffeological Riemannian space is not guaranteed.
			Moreover, so far we do not have any conditions to guaranty a diffeological gradient which does not result in a trivial diffeological Riemannian space.
			This should not be surprising, since even in the more restrictive setting of infinite dimensional manifolds the existence of a gradient is not guaranteed.
		\end{rem}

	\subsubsection{Towards updates of iterates: A diffeological retraction}
		In order to update the iterates, one generally uses the exponential map. However, the exponential map is an expensive operation in optimization techniques. Thus, one uses so-called retractions to update the iterates.
		
		\begin{Definition}[Diffeological retraction]
			\label{def:Retraction}
			Let $ X $ be a diffeological space.
			A diffeological retraction of $ X $ is a map $ \mathcal{R} \colon TX \rightarrow X $ such that the following conditions hold:
			\begin{itemize}
				\item[(i)] 
				$ \mathcal{R}_{\mid_{T_xX}} (0) = x$
				\item[(ii)] 
				Let $\xi \in T_xX $ and $ \gamma_{\xi} \colon T_0\mathbb{R} \to T_xX,\ t \mapsto \mathcal{R}_{\mid_{T_xX}}(t \xi) $. Then $T_0\gamma_{\xi}(0) = \xi$.
			\end{itemize}
		\end{Definition}
		
		Since  finding a diffeological retraction can be quite challenging, we also define the so-called weak diffeological retraction.
		
		\begin{Definition}
			\label{def:WeakRetraction}
			Let $X$ be a diffeological space and $CX$ the tangent cone bundle of $X$. 
			A \emph{weak diffeological retraction} of  $ X $ is a map $ \mathcal{R} \colon CX \rightarrow X $ such that the following conditions hold:
			\begin{itemize}
				\item[(i)] 
				$ \mathcal{R}_{\mid_{C_xX}} (0) = x$
				\item[(ii)] 
				Let $\xi \in C_xX $ and $ \gamma_{\xi} \colon C_0\mathbb{R} \to C_xX,\ t \mapsto \mathcal{R}_{\mid_{C_xX}}(t \xi) $. 
				Then, $C_0\gamma_{\xi} = \xi$.
			\end{itemize}
			In (ii), the map $C_0\gamma_{\xi}$ is a tangential cone map and given by
			$
			C_0\gamma_{\xi} \colon C_0\mathbb{R} \to C_0X, \mathrm{d}_{\alpha} \mapsto \mathrm{d}_{\gamma_{\xi} \circ \alpha}.
			$
		\end{Definition}

			\begin{rem}
				In \cite{GW2020}, there are not only retractions considered. The definition of a diffeological Levi-Civita connection as well as a proof of uniqueness of such a connection is also given. 
					But similar to the Riemannian gradient a condition for the existence is not given.
			\end{rem}

\subsubsection{First order optimization techniques on diffeological spaces}
	Considering algorithm~\ref{Algo}, \cite{GW2020} came up with a diffeological version of the steepest descent algorithm using a retraction instead of an exponential map. For convenience, we repeat this algorithm below (cf. algorithm \ref{AlgoDiff}).
	
	\begin{algorithm}[H]
		\caption{Steepest descent method on the diffeological space $X$ with Armijo backtracking line search}
		\label{AlgoDiff}
		\begin{algorithmic}
			\vspace{.3cm}
			\STATE{ 
				\textbf{Require:} 
				Objective function $f$ on a diffeological Riemannian space $X$; 
				\\\phantom{\textbf{Require:} }diffeological retraction $\mathcal{R}$ on $X$.}
			\vspace{.1cm}
			\STATE{\textbf{Goal:} 
				Find the solution of $\min\limits_{x\in X}f(x)$.}
			\vspace{.1cm}
			\STATE{\textbf{Input:} 
				Initial data $x_0 \in X$; 
				constants $\hat{\alpha}>0$ and $	{\displaystyle \sigma, \rho \in (0,1)}$ for Armijo \\\phantom{\textbf{Input: }}backtracking strategy}
			\vspace{.3cm}
			\STATE{\textbf{for} $k=0,1,\dots$ \textbf{do}}
			\vspace{.1cm}
			\STATE{ [1] Compute $\text{grad}f(x_k)$ denoting the diffeological shape gradient of $f$ in $x_k$.}
			\vspace{.1cm}
			\STATE{ [2] Compute Armijo backtracking step size: }
			\vspace{.1cm}
			\STATE{\hspace*{1cm} Set $\alpha := \hat{\alpha}$.}
			\STATE{ \hspace*{1cm} \textbf{while} $ f\big(\mathcal{R}_{\mid_{T_{x_k}X}}\left(-t_k \text{grad}f(x_k)\right)\big) > f(x^k)-\sigma\alpha \left\|\text{grad}f(x_k)\right\|^2_{T_{x_k}X}$ }
			\STATE{ \hspace*{1cm} Set $ \alpha :=\rho \alpha $.}
			\STATE{ \hspace*{1cm} \textbf{end while}}
			\STATE{ \hspace*{1cm} Set $t_k:=\alpha$.}
			\vspace{.1cm}
			\vspace{.1cm}
			\STATE{ [3] Set $			x_{k+1}:= \mathcal{R}_{\mid_{T_{x_k}X}}\left(-t_k \text{grad}f(x_k)\right).$}
			\vspace{.1cm}
			\STATE{ \textbf{end for}}
			\vspace{.3cm}
		\end{algorithmic}
	\end{algorithm}

	\begin{rem}
	Algorithm \ref{AlgoDiff} is applied to an example in \cite{GW2020}. For more details about the algorithm we refer the interested reader to the corresponding article.
	\end{rem}

		So far, there is no convergence proof of the algorithm.
		A convergence proof would be of grate interest for the diffeological as well as for the optimization community.
		Research work on diffeological optimization theory needs to be done in the future to 
		generalize shape optimization to a  diffeological setting.

	\section{On mapping spaces}
	\label{sec:MappingSpaces}
	
	\subsection{On the ILB setting}
	
	We consider in this subsection { very particular class of Fréchet Lie groups, that are the projective limit {}{of} Banach Lie groups as infinite dimensional Lie groups, that is, where charts and atlases enjoy also the projective limit property.} We rely heavily on \cite{Om} for terminology.
	
	
We must first note that infinite-differentiability in the Fr\'echet sense is equivalent to smoothness in the diffeological sense when we equip a Fr\'echet space/manifold with the diffeology comprising Fr\'echet infinitely-differentiable parametrisations. { The framework that we develop here is more restrictive, but more simple and hence more accessible, than the one fully defined in \cite{Om}. In the sequel the abreviation ILB means ``Inverse Limit of Banach''.}  

\begin{Definition}[ILB-Manifolds \& Bundles]\label{d:ilb}
	\noindent
	\begin{enumerate}
		
		\item A { Fr\'echet manifold $M$ modeled on a Fr\'echet space $F$ is an } \textbf{ILB-manifold}  if 
		{ \begin{itemize}
				\item 
		there exists a sequence of Banach spaces $(B_n)$ such that
	\begin{itemize}
		\item $\forall n \in \N, $ $B_{n+1}$ is a subset of $B_n,$ dense for the topology of $B_n$ and such that the closed unit ball of $B_{n+1}$ is a compact subset of the closed unit ball of $B_n.$
		\item $F$ is the projective limit of the sequence $(B_n)$ as a topological vector space.
	\end{itemize}
With such assumptions $(F,(B_n)_{n \in \N})$ is called an \textbf{ILB vector space}.
\item There exists 	a family of Banach manifolds $\{M_n\}_{n \in \N}$ in which there is a smooth and dense inclusion $M_{n+1}\hookrightarrow M_n$ for each $n$, and such that there exists a Banach atlas on $M_0$ that restricts to an atlas on $M_n$ for each $n$.
		\end{itemize}} 
		
		\item An \textbf{ILB-map} $f$ between two ILB-manifolds $M$ and $N$ is a smooth map $f\colon M\to N$ along with a family of smooth maps $\{f_n\colon M_n\to N_n\}$ such that $f$ and all $f_n$ commute with all inclusions maps $M\hookrightarrow M_{n+1}\hookrightarrow M_n$ and $N\hookrightarrow N_{n+1}\hookrightarrow N_n$.
		\item An \textbf{ILB-principal bundle} is an ILB-map between two ILB-manifolds $(\pi\colon P\to M,~\pi_n\colon P_n\to M_n)$ such that for each $n$, the map $\pi_n\colon P_n\to M_n$ is a principal $G_n$-bundle where $G_n$ is a Banach Lie group.
		
		\item An ILB-map $F$ between two ILB-principal bundles $\pi\colon P\to M$ and $\pi'\colon P'\to M'$ is an \textbf{ILB-bundle map} if $F_n\colon P_n\to P'_n$ is a ($G_n$-equivariant) bundle map for each $n$.  Note that $F$ induces an ILB-map $M\to M'$.  An ILB-bundle map is an \textbf{ILB-bundle isomorphism} if it has an inverse ILB-bundle map.
		
	\end{enumerate}
\end{Definition}
{ When one replaces the sequence of Banach spaces by a sequence of Hilbert spaces, we replace ILB by ILH (Inverse limit of Hilbert) in the definitions. Let us illustrate this setting with the two motivating examples of this setting. For this, we consider a smooth compact boundaryless manifold equipped with a fixed finite atlas $$\{\phi_k: O_k \subset \R^m \rightarrow M\}_{k \in \N_N}.$$
	\begin{ex}
		The space $C^r(M,\R^n)$ of $r-$times continuously differentiable maps, equipped with the norm 
		$$|| f ||_{r,\{\phi_k\}_{k \in \N_N}} = \sup_{k \in \N_N} \sup_{i \in \N_r} \sup_{x \in M} ||D^i f \circ \phi_k||_{L^i(\R^m,\R^n)}$$
		is a Banach norm on $C^r(M,\R^n)$ which does not topologically depend on the choice of the atlas $\{\phi_k\}_{k \in \N_N}.$ Then the sequence $(C^r(M,\R^n))_{r \in \N}$ has $C^\infty(M,\R^n)$ as a projective limit. 
		
		Let us now consider $N$ a $n-$dimensional (locally compact and paracompact) Riemannian manifold. Following \cite{Ee}, the exponential map on $E$ defines pointwise an atlas on $C^0(M,N),$ with domain in $C^0(M,\R^n),$ which by restriction of the domain, defines a structure of smooth Banach manifold on $C^r(M,N),$ modeled on $C^r(M,\R^n),$ for $r \in \N \cup \{\infty\}.$ This defines an ILB manifold. If moreover we replace $G$ by a Lie finite dimensional group, we get an ILB Lie group.  
	\end{ex}

Let us keep the notations of the last example. There always exist a smooth embedding of $N$ into $\R^{2n+1}$ following e.g. \cite{Hir}. We now assume that this is a Riemannian embedding, that is, the Riemannian metric on $N$ is the pull-back of the Riemannian (Euclidian) metric on $\R^{2n+1}.$ We also assume that $M$ is Riemannian. Therefore, for the Riemannian volume $dx$ on $M,$ the space $L^2(M,\R^{2n+1}),$ defined as the completion of $C^\infty(M,\R^{2n+1})$ for the norm 
$$ ||f||_{L^2} = \int_M f^2 dx$$
is a Hilbert space. Moreover, if $\Delta$ is the (positive) Laplacian on $M,$ one can define the Sobolev spaces $H^s(M,\R^{2n+1})$ for $s \in \R$ as the completion of $C^\infty(M,\R^{2n+1})$ for the norm 
$$ ||f||_{H^s} = \int_M\left( (1 + \Delta)^{s/2} f \right)^2 dx,$$
where the real power $(1 + \Delta)^s$ is defined as the operator (unbounded for $s>0$) with eigenvalue $(1 + \lambda_u)^s$ for the eigenvector $u$ of $\Delta$ with respect to the eigenvalue $\lambda_u \geq 0.$ For completeness, we have to recall that $\Delta$ extends to an unbounded, self-adjoint and positive operator on $L^2(M,\R^{2n+1}),$ with smooth eigenvectors, which explains the deep ease for taking some real powers of $(1 + \Delta).$
With these definitions, a special case of the so-called Sobolev embedding theorems is:

\begin{Theorem}
	{}{The inclusion map} $ H^{s+k}(M,\R^{2n+1}) \subset C^k(M,\R^{2n+1}),$ {}{defined $\forall s > m/2,$} is a bounded map and  $H^{s+k}(M,\R^{2n+1})$ is a dense subset of $C^k(M,\R^{2n+1})$ for the $C^k-$topology.
\end{Theorem}

We can now turn to our second example, following again \cite{Ee} as a main reference, completed by \cite{Om}.

\begin{ex} \label{ex:Sobolev}
	Let $s > m/2.$ The set $$H^s(M,N) = H^s(M,\R^{2n+1}) \cap C^0(M,N)$$
	is a (smooth) Hilbert manifold modeled on $H^s(M,\R^n),$ and this family defines ILH structures on $C^\infty(M,N).$ The atlas on $H^s(M,N)$ is the restriction of the atlas on $C^0(M,N)$ where charts are defined pointwise by the exponential map on $N.$ As a consequence, $H^s(M,N)$ is also the closure of $C^\infty(M,N)$ in $H^s(M,\R^{2n+1}).$
\end{ex}

After these two examples, we have here to insist on the fact that there exist examples where {refined properties} of Fr\'echet/Banach manifold require more attention. One can see for instance the manifold of paths defined in \cite{St2017}. 

We now turn to a first point where diffeologies invite themselves in the ILB setting.For this, we consider a ILB space $(F, (B_n)_{n \in \N}).$ The linear groups $GL(B_n)$ are Banach Lie groups, as an immediate consequence of the implicit functions theorem on the Banach algebra $L(B_n)$ of bounded linear maps on $B_n,$ while  example \ref{hypo} explains that the use of a diffeology for $GL(F)$ seem to be more natural, even if it differs from the subset diffeology inherited from the functional diffeology of $L(F).$ 

This problem has been highlighted in the first works concenring ILB spaces. Indeed, in \cite{Om}, H. Omori considers the spaces
$$ \mathcal{L}_n=\cap_{k = 0}^n L(B_k)$$
and invokes ``natural differentiation" on $$\mathcal{L}_\infty = \cap_{n \in \N} \mathcal{L}_k$$ as well as on $$G\mathcal{L_\infty} = \cap_{n \in \N} GL(B_k).$$
Obviously, $G\mathcal{L_\infty}$ cannot be equipped with any atlas modeled on $\mathcal{L}_\infty$ in the actual state of knowledge, that is, we cannot construct explicitely any such atlas with the actual known techniques. Therefore, the extension of the definition of e.g. connections to an ILB setting require many refinements that can be found in \cite{DGV}. 

In the context of diffeologies, we have another way to consider smoothness of  $G\mathcal{L_\infty}$ through the following result: 

\begin{Theorem} \cite{Ma2013}
	Let $(G_n)_{n \in \N}$ be a sequence Banach Lie groups, decreasing for $\subset,$ for which inclusion maps are smooth. Then $G = \bigcap_{n \in \N} G_n$ is a regular Fr\"olicher Lie group.
\end{Theorem}    

This theorem obviously applies to $G\mathcal{L_\infty}.$

}

	\subsection{Mappings with non compact source}

	 {

	 	We first assume that $E$ is a vector bundle with typical fiber $V.$
	 	Let us first consider a fixed local section $s$ of $E,$ with (open) domain $D.$ For any open subset of $M$ such that $ \bar{O} \subset D, $ the evaluation map is smooth. 
	 	Since $M$ is paracompact, one can define a smooth map $\phi: M \rightarrow \R_+$ such that 
	 	$$ \phi(x) = \left\{ \begin{array}{cl} 1 & \hbox{ if } x \in O \\
	 		0 & \hbox{ if } x \notin D \end{array} \right.$$
	 	thus, talking about smoothness with respect to local section, the same is to talk about smoothness with respect to global sections. For the same reasons, one can assume $M$ to be Riemannian. These remarks combined with the existence of local trivialisations on $E,$ we wish to characterize all the topologies on smooth sections of $E, $ that we note by $C^\infty(M,E),$ such that the evaluation maps
	 	$$ \operatorname{ev} : ( x,s) \in M \times C^\infty(M,E) \mapsto \operatorname{ev}_x(s)= s(x) \in E$$
	 	is continuous and moreover, for a given local trivialization $E|_U \sim U \times V$ of $E$ over the open subset $U$ of $M,$ and $\forall \alpha \in \N^*,$ the map
	 	$$ \operatorname{ev}^{(\alpha)}_{U}:  ( x,s) \in M \times C^\infty(M,E) \mapsto \operatorname{ev}^{(\alpha)}_{U,x}(s)= D^\alpha_x s \in L^\alpha( {\R^m},E)$$
	 	is continuous. 
	 	For this purpose, the weaker topology is the pull-back topology on $C^\infty(M,E)$ through the family of maps $\left\{ \operatorname{ev}, \operatorname{ev}_U^{(\alpha)}\right\}.$  This is a standard exercise to show that this topology coincides with the topology defined, for a locally finite family $\{ U_i \, | \, i \in I\}$ of open subset of $M$ over which $E$ is trivial, by the semi-norms
	 	$||.||_{\alpha,K}$ indexed by $\alpha \in \N$ and by the compact subsets $K$ of $M$ defined, for $s \in C^\infty(M,E),$ by 
	 	$$ ||s||_{0,K}  = \sup_{x \in K} ||s(x)||$$
	 	and, for $\alpha >0,$
	 	
	 	$$ ||s||_{\alpha,K}  = \sup\left\{\sup_{x \in K \cap U_i} ||  D^\alpha s(x)|| \, | \, U_i \cap K \neq \emptyset \right\}.$$
	 	In other words, this is the topology of uniform convergence of derivatives at any order of sections $s,$ on any compact subset $K$ of $M.$ With this topology, no global estimate on $S$ is assumed. 
	 	
	 	Let us now assume that $E$ is a fiber bundle with typical fiber $F$ which is a smooth (finite dimensional) manifold for dimension $k.$ Let $\{V_j \, | \, j \in J\}.$ be a locally finite open cover of $F$ such that each open domain $V_j$ is of compact closure in $F.$ Each $V_j$ is identified (by the chosen atlas on $F$) with an open subset $\tilde V_j$ of $\R^k.$ Let $U_i \times F$ be a  local trivialization of $E.$ then on $\operatorname{ev}^{-1}(V_j) \subset U_i \times C^\infty(U_i, F)$ one can consider the distances $d_{i,j,K,\alpha}$ defined by the semi-norms $||.||_{\alpha,K}$ on (local) smooth sections with domain in $U_i$ with values in $V_j \sim \tilde V_j \subset \R^k.$ 
	 	By the way, $C^\infty(M,E)$ \textbf{is a metrizable topological space} for the weaker topology which makes $\left\{ \operatorname{ev}, \operatorname{ev}_U^{(\alpha)}\right\}$ be a family of continuous maps. \textbf{However, if $M$ is not compact, there is actually no atlas for which $C^\infty(M,N)$ is a $C^0-$manifold for this topology} { in {}{contrast} with the case when $M$ is compact described before.}
	 	In order to circumvent this problem, there are two ways: 
	 	
	 	\begin{itemize}
	 		\item consider an exhaustive sequence of compact subsets of $M$, that we note by $(K_n)_{n \in \N}$ and remark that the topology of $C^\infty(M,E)$ is the inductive limit of the Fr\'echet manifold topologies of the increasing family  $$\left\{C^\infty(K_n, E|_{K_n})\, | \, n \in \N \right\}.$$
	 		\item consider the Fr\"olicher structure on $C^\infty(M,E)$ that makes the jet maps $$j^k : f \in C^\infty(M,E) \rightarrow J^k(M,E), \hbox{ for } k \geq 1$$
	 		and 
	 		$$j^\infty : f \in C^\infty(M,E) \rightarrow J^\infty(M,E)$$ smooth { (see \cite{GG1973} for basic definitions on jet spaces, or next section).}
	 	\end{itemize}    
	 	Other stronger topologies exist in $C^\infty(M,N).$ In order to be not so incomplete, we have to mention here the smooth Whitney topology \cite{GG1973}, which is also based on Jet spaces, and for which $C^\infty(M,N)$ is a complete locally convex smooth manifold \cite{KM}. 

	 	\subsection{Jets with non-compact source and a digression on geometric theory of PDEs}
	 	Following \cite[Appendix 2]{MRR-jets}, {the last structures on $C^\infty(M,E)$ extend to $J^k(M,E)$ and to $J^\infty(M,E).$ 
	 		For this, we consider a finite dimensional fiber bundle $E$ with typical fiber $F$ (with $\operatorname{dim}(F)=n$) over a smooth manifold $M$ of dimension $m,$ which can be non compact. In the exposition, we also consider a generic trivialization $\varphi:U \times F \rightarrow E$ of $E$ over an open subset $U$ of $M.$ 
 		The manifold $M$ is the space of 
 	independent variables $x_{i}$, $1 \leq i \leq n$, and the typical fiber is the
 space of dependent variables $u^{\alpha}$, $1 \leq \alpha \leq
m$. The reader may assume that the bundle $\pi$ is simply
${\mathbb R}^n \times {\mathbb R}^{m} \rightarrow {\mathbb R}^n$: even if it looks like we are
``trivializing"  everything because we are working locally, the
fact that we are interested in properties of differential equations and their solutions makes geometry
highly non-trivial, as we show below. 

Let $s_{1}(x_{i})=(x_{i}, s_{1}^{\alpha}(x_{i}))$ and
$s_{2}(x_{i})=(x_{i}, s_{2}^{\alpha}(x_{i}))$ be two local
sections of the bundle $E \rightarrow M$ defined about a point $p =
(x_{i})$ in $M$. We say that $s_{1}$ and $s_{2}$ {\em agree to
order $k$ at $p \in M$} if $s_{1}$, $s_{2}$ and {\em all} the
partial derivatives of the sections $s_{1}$ and $s_{2}$, up to
order $k$, agree at $p$. This notion determines a
coordinate-independent equivalence relation on local sections of
$E$. We let $j^{k}(s)(p)$ represent the equivalence class of the
section $s$ at $p$, and we call this equivalence class {\em the
$k$--jet of $s$ at $p$}. 

\begin{Definition}
The {\em $k$--order jet bundle of $E$} is the space
\[
J^{k}E = \bigcup_{p \in M} J^{k}(p) \; ,
\]
in which $J^{k}(p)$ denotes the set of all the $k$--jets $j^{k}(s)(p)$ of local sections $s$ at $p$.
\end{Definition}

The jet bundle $J^{k}E$ possesses a natural manifold structure; it
fibers over $J^lE$ ($l < k$) and also over $M$: local coordinates
on $J^{k}E$ are $(x_{i}, u^{\alpha}_{0}, u_{i_{1}}^{\alpha},
u_{i_{1}i_{2}}^{\alpha}, \dots ,u_{i_{1}\dots i_{k}}^{\alpha})$,
in which
\begin{equation*}  \label{coo}
u^{\alpha}_{0}(j^{k}(s)(p)) = s^{\alpha}(p) \; , \; \; \; \; \;
u_{i}^{\alpha}(j^{k}(s)(p)) = \frac{\partial s^{\alpha}}{\partial
	x_{i}}(p) \; , \; \; \; \; \; \; \;
u_{i_{1}i_{2}}^{\alpha}(j^{k}(s)(p)) = \frac{\partial^{2}
	s^{\alpha}}{\partial x_{i_{1}}
	\partial x_{i_{2}}}(p) \; ,
\end{equation*}
and so forth, where $j^{k}(s)(p) \in J^kE$ and $(x^i) \mapsto (x^i , s^\alpha(x^i))$ is any local section in the 
equivalence class $j^k(s)(p)$.
In these coordinates, the projection map {
$\alpha^{k} : J^{k}E \rightarrow M$ (source map) factors through a projection map $\pi^k : J^k(E) \rightarrow E$ for which, on a local generic trivialization $\varphi$ as defined before, where $E \sim U \times F$ locally, $\pi^k = \alpha^k \times \beta^k$ where $\beta^k$ is the local target map. Moreover, the projections} 
$\pi^{k}_{l} : J^{k}E \rightarrow J^{l}E$, $l < k$, are defined in obvious ways.
For instance, $\pi_{M}^{k}$ is simply $(x_{i}, u^{\alpha}_{0},
u_{i_{1}}^{\alpha}, \dots , u_{i_{1}\dots i_{k}}^{\alpha}) \mapsto
(x_{i})$. 
{
Let us precise a little more: for $p \in E,$ $(\pi_k)^{-1}(p)$ is a vector space modelled on 
$$\bigoplus_{i = 1}^k L^i(\R^m,\R^n)$$
where $L^i(\R^m,\R^n)$ is the space of $i-$linear symmetric mappings on $\R^m$ with values in $\R^n.$ 
Therefore, $J^k(M,E)$ can be understood as:
\begin{itemize}
	\item a finite dimensional vector bundle over $E$ with typica fiber $\bigoplus_{i = 1}^k L^i(\R^m,\R^n)$
	\item a finite dimensional fiber bundle over $M$ with typical fiber $$F \times \left(\bigoplus_{i = 1}^k L^i(\R^m,\R^n)\right).$$
\end{itemize}An extended exposition of the properties of $k-$jets can be found in \cite{GG1973}.}

\smallskip

Any local section $s : (x_{i}) \mapsto (x_{i},s^{\alpha}(x_{i}))$ of $E$ lifts to a unique local section 
$j^{k}(s)$ of $J^{k}E$ called the {\em $kth$ prolongation} of $s$. In coordinates, $j^{k}(s)$ is the section
\[
j^{k}(s) = \left( x_{i}, s^{\alpha}(x_{i}), \frac{\partial
s^{\alpha}}{\partial x_{i_{1}}} \, (x_{i}) \, , ...,
\frac{\partial^{k} s^{\alpha}}{\partial x_{i_{1}} \dots
\partial x_{i_{k}}} \, (x_{i}) \, , ... \right) \; .
\]
The {\em infinite jet bundle} $\pi^\infty_M : J^{\infty}E \rightarrow M$ is the
inverse limit of the tower of jet bundles  $M \leftarrow E \cdots \leftarrow J^{k}E \leftarrow J^{k+1}E \leftarrow
\cdots$ under the
standard projections $\pi^{k}_{l} : J^{k}E \rightarrow J^{l}E$, $k > l$. 
{
We have here to precise that we consider the inverse limit and not the inductive limit. That is, comparing with $J^k(M,E)$ which is a finite dimensional vector bundle over $E,$ taking the inverse limit, we also get $J^\infty(M,E)$ as a vector bundle over $E$ with typical fiber the \textbf{formal series} $$ \sum_{i = 1}^{+\infty} L^i(\R^m,\R^n).$$	
}We describe the space $J^{\infty}E$ locally, by sequences
\begin{equation}  \label{co}
(x_{i}, u^{\alpha}, u_{i_{1}}^{\alpha}, ..., u_{i_{1}i_{2}\dots
	i_{k}}^{\alpha}, ...) \; , \; \; \; 1 \leq i_1 \leq i_2 \leq \dots
\leq i_k \leq n \; ,
\end{equation}
obtained from the standard coordinates on the finite--order jet
bundles $J^{k}E$.

The limit $J^\infty E$ is a topological
space: a basis for the topology on $J^\infty E$ is the collection of all sets of the form 
$(\pi^{\infty}_{k})^{-1}(W)$, in which $W$ is an
open subset of $J^{k}E$.   We would like to consider $J^\infty E$ as a manifold which needs to be at least an infinite dimensional manifold/ we need to re-define all of the standard differential geometry notions
in this new context, see \cite{A,AK}:

Let $\pi^{\infty}_{k} : J^{\infty}E \rightarrow J^{k}E$ denote the canonical projection from $J^{\infty}E$ onto 
$J^{k}E$. A function $f : J^{\infty}E \rightarrow {\mathbb R}$ is {\em smooth} if it factors through a finite-order 
jet bundle, that is, if $f = f_{k} \circ \pi^{\infty}_{k}$ for some smooth function $f_{k} : J^{k}E \rightarrow 
{\mathbb R}$. Real-valued smooth functions on $J^\infty E$ are also called {\em differential functions}, see 
\cite{O}. {

\textbf{This notion of smoothness coincides with the notion of Fr\"olicher space, where a generating set of functions $\F_0$ serves to define which set of mapping is smooth.} 
Moreover, considering the typical fibers 
$\sum_{i = 1}^{+\infty} L^i(\R^m,\R^n)$ of $J^\infty(M,E)$ as a vector bundle over $E,$ since each $L^i(\R^m,\R^n)$ is finite dimensional, it us easy to prove that  $\sum_{i = 1}^{+\infty} L^i(\R^m,\R^n)$ is a Fr\'echet space, equipped with the semi-norms $$||.||_i =|| . ||_{L^i(\R^m,\R^n)} \circ \pi_i$$ where $\pi_i:\sum_{i = 1}^{+\infty} L^i(\R^m,\R^n) \rightarrow  L^i(\R^m,\R^n)$ is the canonical projection.

When one tries to define the tangent space $TJ^\infty(M,E),$ it is natural to consider all derivations $X$ on $\F_0.$ 
In this very particular setting, by restriction of the derivation $X$ to $C^\infty(J^k(M,E),\R )\subset \F_0,$ since $J^k(M,E)$ is a finite dimensional manifold, then the restriction of $X$ to $J^k(M,E)$ can be expressed as a (classical) vector field  over $J^k(M,E)$ and, by inverse limit procedure, $TJ^\infty(M,E)$ is also the inverse limit of the sequence $$TJ^1(M,E) \leftarrow... \leftarrow TJ^k(M,E) \leftarrow ...$$ which shows that Olver's definition of $TJ^\infty(M,E)$ coincides with the standard definition of the tangent space to the Fr\'echet manifold $J^\infty(M,E).$
}


	 		These structures are useful to talk about partial differential equations understood as \textbf{partial differential relations} along the lines of e.g. \cite{Gro1986}. Indeed, a basic partial differential equation of the form $A(u)=0,$ where $A$ is a linear differential operator of order $k,$ envolves a solution $u$ and its partial derivatives till the order $k,$ and hence can be encoded as $\Xi (j^k(u))=0$ where $$\Xi: J^k(M,\R) \rightarrow \R$$ is a smooth map. The same holds for most non-linear partial differential equations of interest, but it can occur that the space $ \Xi^{-1}(0)$ can be equipped with many ``singularities'' as shows this very non-linear example based on a non-linear differential equation:}
	 	
	 	
	 	{
	 		\begin{ex} \label{piecewiseaffine}
	 			When $\Xi(x,y,y')= x\cos(y)\sin(4y'),$ the subset $\Xi^{-1}(0)$ has an 
	 			infinity of singular points. Let us assume that we restrict $\mathcal{E}$ to a neighbourhood $U$ of a regular point. 
	 			The solutions $x \mapsto (x,y(x))$ we would find are simply constant solutions or affine solutions and we may lose 
	 			information depending on our choice of $U$. Now, solutions extend to piecewise affine (and hence only piecewise 
	 			smooth) solutions on $\Xi^{-1}(0)$. 
	 	\end{ex}}
	 	The problem of the geometry of singularities needs the definition of a generalized way to define differential geometric objects. 
	 	
	 	{  
	 	
	 	{ Let us now consider a partial differential equation $\Xi = 0$ in $J^k(M,E) \subset J^\infty(M,E).$ Following Example \ref{piecewiseaffine}, this is a bit too optimistic to hope that what we call here \emph{the vanishing set} along the lines of \cite[Appendix 2]{MRR-jets}, which is the subset $\Xi^{-1}(0),$ has a structure of smooth manifold. Therefore, in the classical literature, one invokes the use of a \emph{locus}, that is, a smooth submanifold of $J^k(M,E)$ with range in $\Xi^{-1}(0).$ Typically, this locus can be defined by an unique embedding of an open subset of an Euclidian space $O$ to $J^k(M,E).$ This enables to define tangent space, differential forms and other objects on $\{\Xi=0\},$ actually with a slightly non rigorous arguments concerning their global definition on the full space $\Xi^{-1}(0).$ until \cite{MRR-jets} in the context of diffeologies. Indeed, the subset diffeology defines in a natural way a diffeological structure on the vanishing set, which completes the basic setting for the geometry on the main basic example of \textit{diffieties} given by $\Xi^{-1}(0),$ along the lines of \cite{MRR-Vin}. 
	 	
 \begin{rem}
 	The spaces of jets are also useful for alternate definitions of connections and their generalizations, see \cite{KMS}, {and an application of diffeologies to higher connections may be feasible.} We do not have enough place here to develop this aspect of the theory.
 \end{rem}	
 }

	 \subsection{Mappings with low regularity}
	 
	{ In the sequel we consider an algebra $\mathcal{M}$ with complex coefficients, with the hermitian product of matrices $(A,B) \mapsto tr(AB^*).$
	 If there is no possible confusion, we note this matrix norm by $||.||$ or by $||.||_{\mathcal{M}}$ if necessary.
	 We now extend the construction of Example \ref{ex:Sobolev} the following way, along the lines of \cite{Ma2019}.
	  \begin{Definition}
	  	Let $s \in \R,$ let  $M$ be a compact boundaryless Riemannian manifold and let $N$ be a Riemannian manifold embedded in $\mathcal{M}.$
	 	We define $H^s(M,N),$  as the completion of $C^\infty(M,N)$ in $H^s(M,\mathcal{M}).$ 
	 \end{Definition}
 We remark that the condition $s < \operatorname{dim}(M)/2$ is not assumed here. This implies a weaker differential structure on $H^s(M;\mathcal{M})$ when $s \leq dim(M) /2$ which can be viewed as follows, generalizing \cite{Ma2019} in a straightforward way: 
 \begin{Proposition} 
 	$H^s(M,N)$ is a diffeological space, equipped with the subset diffeology inherited from $H^s(M,\mathcal{M}).$
 \end{Proposition} }
	 Let us now consider a compact connected Lie group $G$ of matrices.
	 
	 For $s>\operatorname{dim}(M)/2,$  it is well-known, that $H^s(M,G)$ is a Hilbert Lie group. 
	The case where $M=S^1,$ that is the loop space or the loop group, is the most widely studied in the literature, see e.g. \cite{PS}. The biggest Sobolev order where $H^s(S^1,G)$ fails to be a Hilbert manifold, and also a group, is $s=1/2.$
	 
	 \begin{Proposition} 
	 	Let $s \leq 1/2.$. Then $H^{s}(S^1,G)$ and $H^{s}_0(S^1,G)$ are Fr\"olicher space.
	 \end{Proposition}

This proposition completes the classical setting, in which the study of the $H^{1/2}-$metric is a key point, from a geometric viewpoint, see e.g. \cite{F2,PS}, as from the viewpoint of shape analysis, see e.g. \cite{We}. 
	 \subsection{Groups of diffeomorphisms are regular... or not}
	 
	 Let $M$ be a locally compact, non-compact manifold,
	 {
	 	which is assumed to be Riemannian without restriction,} equipped with its n\'ebuleuse diffeology. We equip the group of diffeomorphisms $\operatorname{Diff}(M)$ with the topology of convergence of the derivatives an any order, uniformly on each compact subset of $M,$ usually called $C^\infty-$compact-open topology or weak topology in \cite{Hir}. Traditionnally, $\operatorname{Vect}(M)$ is given as the Lie algebra of $\operatorname{Diff}(M),$ but \cite[section 43.1]{KM} shows that this strongly depends on the topology of $\operatorname{Diff}(M).$ 
	 {
	 	Indeed, the Lie algebra of vector fields described in \cite[section 43.1]{KM} is the Lie algebra of compactly supported vector fields, which is not the (full) Lie algebra $\operatorname{Vect}(M).$
	 	In another context, when $M$ is compact, $\operatorname{Vect}(M)$ is the Lie algebra of $\operatorname{Diff}(M),$ which can be obtained by Omori's regularity theorems \cite{Om1973,Om} and recovered in \cite{CW}.
	 } 
	 What is well known is that infinitesimal actions 
	  of $\operatorname{Diff}(M)$ on $C^\infty(M,\R)$ generate vector fields, viewed as order 1 differential operators.
	 {
	 	The bracket on vector fields is given by $$(X,Y)\in \operatorname{Vect}(M)	\mapsto [X,Y]= \nabla_XY - \nabla_YX,$$
	 	where $\nabla$ is the Levi-Civita connection on $TM.$ This is a Lie bracket, stable under the Adjoint action of $\operatorname{Diff}(M).$
	 } Moreover, the compact-open topology on $\operatorname{Diff}(M)$ generates a corresponding $C^\infty-$compact-open topology on $\operatorname{Vect}(M).$ This topology is itself the $D-$topology for the the {
	 	functional} diffeology on $\operatorname{Diff}(M).$  Following \cite[Definition 1.13 and Theorem 1.14]{Les}, $\operatorname{Vect}(M)$ equipped with the $C^\infty$ compact-open topology is a Fr\'echet vector space, {
	 	and  the Lie bracket is smooth.
	 	Moreover, we feel the need to remark that}  the evaluation maps
	 $$ T^*M \times Vect(M) \rightarrow \R$$
	 separate $\operatorname{Vect}(M).$ Thus $\operatorname{Diff}(M)$ is a diffeological Lie group matching with the criteria of \cite[Definition 1.13 and Theorem 1.14]{Les}.
	
	 Let $F$ be the vector space of smooth maps $f \in C^\infty(]0;1[;\R).$ We equip $F$ with the following semi-norms:
	 
	 For each $(n,k) \in \N^* \times \N $,  $$||f||_{n,k} = \sup_{\frac{1}{n+1}\leq x \leq \frac{n}{n+1}} |D^k_xf|.$$
	 
	 This is a Fr\'echet space, and its topology is the smooth compact-open topology, which is the $D-$topology of the compact-open diffeology. 
	 Let 
	 $$\A = \{ f \in C^\infty(]0;1[;]0;1[)|\lim_{x \rightarrow 1}f(x)=1
	 \wedge\lim_{x \rightarrow 0}f(x)=0 \}.$$
	 Finally, we set
	 $$\D = \{ f \in \A | \inf_{x \in ]0;1[}f'(x) >0 \}.$$
	 $\D$ is a {
	 	contractible} set of diffeomorphisms of the open interval $]0;1[$ which is an (algebraic) group for composition of functions. Composition of maps, and inversion, is smooth for the {
	 	functional} diffeology.
	 Unfortunately, $\D$ is not open in $\A.$
	 As a consequence, we are unable to prove that it is a Fr\'echet Lie group. However, considering the smooth diffeology induced on $\D$ by $\A,$ the inversion is smooth. As a consequence, $\D$ is (only) a diffeological Lie group.
	 
	%
	 %


{
}

Following \cite{Ma2018-2}, we can state:
\begin{Theorem}
$\D$ is a non-regular Fr\"olicher Lie group. 
\end{Theorem}

We now give a consequence of this first theorem for the non-regularity of the group of diffeomorphisms of a non-compact boundaryless manifold, following \cite{Ma2018-2}.
}

\begin{Theorem} 
 $Diff(M)$ is a non-regular Fr\"olicher Lie group.
\end{Theorem}

{ \begin{rem}
		In this statement and as we stated before, $\operatorname{Diff}(M)$ is equipped with a topology which is called weak topology in some parts of the literature. It appears to be one of the weakest topologies, inherited from a Fr\"olicher structure, that make evaluation maps continuous and smooth. 
		This is in great contrast with the large class of examples of manifold structures on $\operatorname{Diff}(M)$ for which $\operatorname{Diff}(M)$ is regular \cite{KM,KMR}. Therefore, as in the context of loop spaces described before, the geometry and the topology depend a lot on the chosen geometric structures.    
	\end{rem}}
 
	\subsection{Geometry of finite and infinite triangulations}
The notion of triangulation is itself a bit confuse in the litterature. One can define: 
\begin{itemize}
	\item either parametrized smooth triangulations in a manifold $M$: 
	\begin{Definition} \label{smtriang} A \textbf{smooth triangulation} of $M$ is a family $\tau = (\tau_i)_{i \in I}$
		where $I \subset \N$ is a set of indexes, finite or infinite, each $\tau_i$ is a smooth map $\Delta_n \rightarrow M,$ and such that:
		\begin{enumerate}
			\item $\forall i \in I, \tau_i$ is a (smooth) embedding, i.e. a smooth injective map such that {{}$(\tau_i)_*\left(\p({\Delta_n})\right)$} is also the subset diffeology of $\tau_i(\Delta_n)$ as a subset of $M.$
			\item $\bigcup_{i \in I }\tau_i(\Delta_n) = M.$ (covering)
			\item $\forall (i,j) \in I^2,$ $\tau_i(\Delta_n) \cap \tau_j(\Delta_n) \subset  \tau_i(\partial \Delta_n) \cap \tau_j(\partial \Delta_n).$ (intersection along the borders)
			\item  $\forall (i,j) \in I^2$ such that $ D_{i,j}=\tau_i(\Delta_n) \cap \tau_j(\Delta_n) \neq \emptyset,$ for each $(n-1)$-face $F$ of $D_{i,j},$ the ``transition maps" $``\tau_j^{-1} \circ \tau_i'' : \tau_i^{-1}(F) \rightarrow \tau_j^{-1}(F)$ are affine maps.  
		\end{enumerate} 
	\end{Definition} 
\item or non-parametrized smooth triangulations, which consider only the images in $M$ of the tarndard simplex through a smooth simplex of the triangulation. 
\end{itemize}
In open domains $\Omega \subset \R^n, $ one often restricts to {\bf affine} triangulations, for which, for each $k \in \N_n,$  $(n-k)-$ dimensional faces lie in $(n-k)-$dimensional affine subspaces of $\R^n$ and, in the parametrized case, embeddings are (restrictions of) affine maps. 
	\subsubsection{Diffeologies associated to a fixed triangulation}
	There exists actually many diffeologies on a fixed triangulated manifold, and more generally, on a CW-complex and the first described were in { \cite{Nt2002,Nt2005} many years before \cite{CW2014}, see e.g. \cite{Kih2019,Kat2020,Kat2021} while other approaches are still in progress.} Each of them has its own technical interest. We give here a selected one adapted to our needs. we begin with a lemma from \cite{Ma2016-2} which is adapted from so-called gluing results present in \cite{Nt2002,pervova2017,pervova2018} to the context which is of interest for us.
	\begin{Lemma}  \label{cov}
		Let us assume that $X$ is a topological space, and that there is a collection $\{(X_i,\F_i,\mcc_i)\}_{i \in I}$ of Fr\"olicher spaces, 
		together with continuous maps $\phi_i: X_i \rightarrow X.$
		Then we can define a Fr\"olicher structures on $X$
		setting $$\F_{I,0} = \{f \in C^0(X,\R) | \forall i \in I, \quad f \circ \phi_i \circ \mcc_i \subset C^\infty(\R,\R)\},$$ wa define $\mcc_I$ the contours generated by the family $\F_{I,0},$ and  $\F_I = \F(\mcc_I).$
	\end{Lemma}
	
	Let $M$ be a smooth manifold for dimension $n.$ Let \begin{equation} \label{embtriangulation}\Delta_n= \{(x_0,...,x_n)\in \R_+^{n+1} | x_0 +... +x_n = 1\}\end{equation} be the standard  $n-$ simplex, equipped with its subset diffeology. It is easy to show that this diffeology is reflexive through Boman's theorem already mentionned, and hence we can call it Fr\"olicher space $(\Delta_n, \F_{\Delta_n}, \C_{\Delta_n}),$ and we {{} denote} its associated reflexive diffeology by $\p(\Delta_n).$ 
	 	Under these  conditions, we equip the triangulated manifold $(M,\tau)$ with a Fr\"olicher structure $(\F_I,\mcc_I),$ generated by the smooth maps $\tau_i$ applying Lemma \ref{cov}. The following result is obtained from the construction of $\F$ and $\mcc:$
	 \begin{Theorem}
	 	The inclusion $ (M,\F,\mcc) \rightarrow M$ is smooth.
	 \end{Theorem}
	\subsubsection{Geometry of the space of triangulations}
	Let us now fully develop an approach based on the remarks given in  \cite{Ma2016-2}. For this, the space of triangulations of $\Omega$ is considered itself as a Fr\"olicher space, and the mesh of triangulations which makes the finite element method converge will take place, as the function $f,$  among the set of parameters $Q.$ We describe here step by step the Fr\"olicher structure on the space of triangulations.  
	By the way,

	Under these  conditions, we equip the triangulated manifold $(M,\tau)$ with a Fr\"olicher structure $(\F_I,\mcc_I),$ generated by the smooth maps $\tau_i$ applying Lemma \ref{cov}. The following result from \cite{Ma2020-3} is obtained from the construction of $\F$ and $\mcc:$
	\begin{Theorem} 
		The inclusion $ (M,\F,\mcc) \rightarrow M$ is smooth.
	\end{Theorem}
	
		
	
	\begin{rem} Maps in $\F_I$ can be intuitively identified as some piecewise smooth maps $M \rightarrow \R,$ which are of class $C^0$ along the 1-skeleton of the triangulation.
		We have proved also that $\mcc_I \subset \p_\infty(M).$ Some characteristic elements of $\mcc_I$ can be understood as paths which are smooth (in the classical sense) on the interiors of the domains of the simplexes of the triangulation, and that fulfill some more restrictive conditions while crossing the 1-skeleton of the triangulation. For example, paths that are (locally) stationnary at the 1-skeleton are in $\mcc_I.$
	\end{rem}
	
	\begin{rem}
		While trying to define a Fr\"olicher structure from a triangulation, one could also consider $$\C_{I,0} = \left\{ \gamma \in C^{0}(\R,M)\, | \, \forall i \in I, \forall f \in C^\infty_c(\phi_i(\Delta_n),\R), f \circ \gamma \in C^\infty(\R,\R) \right\}$$ where $C^\infty_c(\phi_i(\Delta_n),\R)$ stands for compactly supported smooth functions $M \rightarrow \R$ with support in $\phi_i(\Delta_n).$ Then define $$\F_I' = \left\{f : M \rightarrow \R \, | \, f \circ \C_{I,0} \in C^\infty(\R,\R)\right\}$$
		and $$\C_I' = \left\{C : \R \rightarrow M \, | \, \F_I' \circ c \in C^\infty(\R,\R)\right\}.$$
		We get here another construction, but which does not understand as smooth maps $M \rightarrow \R$ the maps $\delta_k$ already mentionned.
	\end{rem}
	Now, let us fix the set of indexes $I$ and fix a so-called \textbf{model triangulation} $\tau.$ {{} This terminology is justified by two ideas: 
		\begin{itemize} 
			\item Anticipating next constructions, this model triangulation $\tau$ will serve at defining a sequence of refined trinagulations. This is our ``starting triangulation'' for the refinement procedure in the finite elements method. 
			\item Changing $\tau$ into $g \circ \tau,$ where $g$ is a diffeomorphism, we get another model triangulation, which has merely the same properties as $\tau.$ But each ``starting'' trinagulation cannot be obtained by transforming a fixed triangulation by using a diffeomorphism. For example, on the 2-sphere, a tetrahedral triangulation $\tau_1$ and an octahedral triangulation $\tau_2$ separately generate two sequences of refined triangulations, and there is a topological obstruction for changing $\tau_1$ into $\tau_2$ by the action of a diffeomorphism of the sphere.     
	\end{itemize} } We {{} denote} by $\mathcal{T}_\tau$ the set of triangulations $\tau'$ of $M$ such that the corresponding 1-skeletons are diffeomorphic to the 1-skeleton of $\tau$ (in the Fr\"olicher category). {{} The set $\mathcal{T}_\tau$ contains, but is not reduced to, the orbit of $\tau$ by the action of the group of diffeomorphisms. Indeed, one can reparametrize each simplex with adequate compatibility on the border. Intuitively speaking, reparametrizations need not to be smooth in the usual sense while ``crossing the border of a simplex''. This choice is motivated by the Fr\"olicher structure that we identify as useful for the finie elements method, ddefined hereafter.}

	\begin{Definition} 
		Since $\mathcal{T}_\tau \subset C^\infty(\Delta_n, M)^I,$ we can equip  $\mathcal{T}_\tau$ with the subset Fr\"olicher structure, in other words, the Fr\"olicher structure on $\mathcal{T}_\tau$ whose generating family of contours $\mcc$ are the contours in $C^\infty(\Delta_n, M)^I$ which lie in $\mathcal{T}_\tau.$
	\end{Definition} 
	We define the full space of triangulations $\mathcal{T}$ as the disjoint union of the spaces of the type  $\mathcal{T}_\tau,$ with disjoint union Fr\"olicher structure. With this notation, in {{}the} sequel and when it carries no ambiguity, the triangulations in $\mathcal{T}_\tau$ is equipped with a fixed set of indexes $I$ (which is impossible to fix for $\mathcal{T}$). We need now to describe the procedure which intends to refine the triangulation and define a sequence of triangulations $(\tau_n)_{n \in \N}.$ 
	We can now consider the refinement operator, which is the operator which divides a simplex $\Delta_n$ into a triangulation. 
	
	\begin{Definition}
		Let $m \in \N,$ with $m \geq 3.$ Let 
		$$\mu = \left\{\mu_i: \Delta_n \rightarrow \Delta_n \, | \, i \in \N_m\ \right\}$$ be a smooth triangulation of $\Delta_n$ 
		Let $\tau \in \mathcal{T}.$ Then we define {{}$$\mu(\tau) = \{f_i \circ \mu_i \, | \,  i \in \N_m \hbox{ and } \tau = (f_i)_{i \in I}\}.$$ We say that $\mu$ defines a \textbf{refinement map} if $\forall n \in \N^*, \mu^n(\tau)$ is a triangulation. }
	\end{Definition}
	
	With this definition, $\mu(\tau)$ is trivially a triangulation of $M$ if $\tau$ is a triangulation of $M.$
	The conditions imposed in the definition ensures that the refinement map maps {{}a} triangulation to another {{}triangulation}, that is, if $\tau$ is a triangulation, $\mu(\tau)$ is {{}also} a {{}triangulation}. The delicate needed condition is that the new 0-vertices added to $tau$ in $\mu(\tau)$ are matching. 

	
	\begin{Definition}
		Let $\tau \in \mathcal{T}.$ We define the $\mu-$refined sequence of triangulations $\mu^\N(\tau) = (\tau_n)_{n \in \N}$ by $$ \left\{ \begin{array}{ccl} \tau_0 & = & \tau \\ \tau_{n+1} & = & \mu(\tau_n) \end{array} \right.$$ 
	\end{Definition}
	
	\begin{Proposition} \label{seqref} 
		The map $$\mu^\N : \mathcal{T} \rightarrow \mathcal{T}^\N$$ is smooth (with $\mathcal{T}^\N$ equipped with the infinite product Fr\"olicher structure).
	\end{Proposition}
%
%
	
	Let $\Omega$ be a bounded connected open subset of $\R^n,$ and assume that the border $\partial \Omega = \bar{\Omega} - \Omega$ is a polygonal curve. Since $\R^n$ is a vector space, we can consider the space of affine triangulations: 
	$$\operatorname{Aff}\mathcal{T}_\tau = \left\{ \tau' \in \mathcal{T}_\tau | \forall i , \tau_i' \hbox{ is (the restriction to } \Delta_n \hbox{ of) an affine map } \right\}.$$
	We define $\operatorname{Aff}\mathcal{T}$ from $\operatorname{Aff}\mathcal{T}_\tau$ the same way we defined $\mathcal{T}$ from $\mathcal{T}_\tau,$ via disjoint union.
	We equip $\operatorname{Aff}(\mathcal{T}_\tau)$ {{}and} $\operatorname{Aff}(\mathcal{T})$ with their subset diffeology. We use here the notations of last Lemma.
	\begin{Theorem}
		Let $$c : \R \rightarrow \operatorname{Aff}(\mathcal{T}_\tau)$$ be a path on $\operatorname{Aff}(\mathcal{T}_\tau).$ Then $$ c \hbox{ is smooth } \Leftrightarrow \forall (i,j) \in I \times \N_{n+1}, t \mapsto x_j(c(t)_i) \hbox{ is smooth. }$$
	\end{Theorem}

	\begin{Proposition}
		Let {{}$\mu$} be a fixed affine triangulation of $\Delta_n.$
		The map $\mu^\N$ restricts to a smooth map from the set of affine triangulations of $\Omega$ to se set of sequences of affine triangulations of $\Omega.$
	\end{Proposition}

\subsection{Diffeologies and implicit functions in the ILB setting}
We set the following notations, 
from the standard reference \cite{Om} and along the lines of \cite{Ma2020-1}:
Let  $\hbox{\bf E} = (E_\infty,(E_i)_{i \in \N})$ and  
$\hbox{\bf F} = (F_\infty,(F_i)_{i \in \N})$ be two ILB vector spaces. 
Let $O_0$ be an open neighborhood of $(0{{},}0)$ in $E_0 \times F_0,$ let
$\hbox{\bf O} = (O_i)_{i \in I}$ with $O_i = O_0 \cap ({E_i \times F_i})$, for $i \in \N \cup \{\infty\}.$ 

{{} Let us now propose, { along the lines of \cite{Ma2020-3}}, a diffeological approach to the main result of \cite{Ma2020-1} that we recall here. For this, we consider a function $f_0$ of class $C^\infty$such that
	\begin{enumerate}
		\item $f_0(0; 0) = 0$
		\item $D_2f_0(0; 0) = Id_{F_0}$.
	\end{enumerate} Moreover, let us assume that $f_0$ restricts to $C^\infty-$maps
	$ f_i : U_i \times V_i \rightarrow F_i,$ {and defines an ILB map $$\mathbf{f}:\mathbf{E} \rightarrow \mathbf{F}.$$}
	let $$U_\infty = E_\infty \cap U_0 \quad \hbox{ and } \quad V_\infty = V_0 \cap F_\infty.$$  
	We do not assume here any other assumption, contrasting with e.g. the classical Nash-Moser theorem \cite{Ham1982} where additional norm estimates are necessary. Under our weakened conditions, one can state \cite[Theorem 2.2]{Ma2020-1}:
	\begin{Theorem} \label{1.6} There exists a non-empty domain $D_\infty \subset U_\infty,$ possibly non-open
		in $U_\infty,$ and a function $$u_\infty : D_\infty \rightarrow V_\infty$$ such that
		$$\forall x \in D_\infty, \quad f_\infty(x; u_\infty(x)) = 0.$$
		{{} Moreover, there exists a sequence $(c_i)_{i \in \N} \in (\R_+^*)^\N$ and a Banach space $B_{f_\infty}$ such that
			\begin{itemize}
				\item $B_{f_\infty} \subset E_\infty$ (as a subset)
				\item the canonical inclusion map $B_{f_\infty} \hookrightarrow E_\infty$ is continuous
			\end{itemize} 
			which is the domain of the following norm (and endowed with it): $$||x||_{f_\infty} = \sup
			\left\{ \frac{||x_i||}{c_i}| i \in\N \right\}. $$ Then $D_\infty$ contains $\mathcal{B},$ the unit ball (of radius $1$ centered at 0) of  $B_{f_\infty}.$} \end{Theorem}   
	In \cite{Ma2020-1}, the question of the regularity of the implicit function is left open, because the domain $D_\infty$ is not a priori open in $O_\infty.$ 
	This lack of regularity induces a critical breakdown in generalizing the classical proof of the Frobenius theorem to this setting. { This gap is filled in \cite{Ma2020-3} mostly by the following result:  }
	\begin{Theorem} \label{IFTh}
		{{}Let $$f_i: O_i \rightarrow F_i, \quad i \in \N \cup \{\infty \}$$ be a family of maps, 
			let $u_\infty$ the implicit function defined on the domain $D_\infty$,  as in Theroem \ref{1.6}.
			Then, there exists a domain $D$ such that $\mathcal{B} \subset D \subset D_\infty$ such that the function $u_\infty$  is smooth for the subset diffeology of $D.$}
	\end{Theorem}

	The same way, we can state the corresponding Frobenius theorem {from the same reference:}
	
	\begin{Theorem}\label{lFrob}
		
		Let 
		$$ f_i : O_i \rightarrow L(E_i,F_i), \quad i \in {{}\N}$$ 
		be a collection of smooth maps satisfying the following condition: 
		$$ i > j \Rightarrow f_j|_{O_i} = f_i$$ and such that, $$\forall (x,y) \in O_i, \forall a,b \in E_i$$
		$$(D_1f_i(x,y)(a)(b) + (D_2f_i(x,y))(f_i(x,y)(a))(b) =$$
		$$(D_1f_i(x,y)(b)(a) + (D_2f_i(x,y))(f_i(x,y)(b))(a) .$$

		Then,
		$\forall (x_0, y_0) \in O_{\infty}$, there exists a diffeological subspace  $ D $ of $O_\infty$ that contains $(x_0, y_0)$ and a smooth map
		$J : D \rightarrow  F_\infty$
		such that
		$$ \forall (x,y) \in D, \quad D_1J(x,y) = f_i(x, J(x,y)) $$
		and, if {{}$D_{x_0}$ is the connected component of $(x_0,y_0)$ in $\{(x,y) \in D \, | \, x = x_0  \},$ }
		$$J_i(x_0,.) = Id_{D_{x_0}}.
		$$ 
		{{} Moreover, there exists a sequence $(c_i)_{i \in \N} \in (\R_+^*)^\N$ and a Banach space $B_{f_\infty}$ such that
			\begin{itemize}
				\item $B_{f_\infty} \subset E_\infty\times F_\infty$ (as a subset)
				\item the canonical inclusion map $B_{f_\infty} \hookrightarrow E_\infty\times F_\infty$ is continuous
			\end{itemize} 
			which is the domain of the following norm (and endowed with it): $$||x||_{f_\infty} = \sup
			\left\{ \frac{||x||_{E_i \times F_i}}{c_i}| i \in\N \right\}. $$ Then $D_\infty$ contains $\mathcal{B},$ the unit ball (of radius $1$ centered at 0) of  $B_{f_\infty}.$}
		
	\end{Theorem}

\end{document}